\documentclass[aap, preprint]{imsart}
\usepackage{cleveref}
\usepackage{braket,amsfonts}
\usepackage{color}
\usepackage{cite}
\usepackage{graphicx}
\usepackage{psfrag}
\usepackage{url}
\usepackage{stfloats}
\usepackage{graphicx}
\usepackage{float}
\usepackage{epstopdf}
\usepackage{amssymb}
\usepackage{array}
\RequirePackage[OT1]{fontenc}
\RequirePackage{amsthm,amsmath}
\RequirePackage[numbers]{natbib}
\RequirePackage[colorlinks,citecolor=blue,urlcolor=blue]{hyperref}

\arxiv{arXiv:0000.0000}

\usepackage{xspace}
\usepackage{bold-extra}

\def\a{\alpha} \def\b{\beta}  \def\d{\delta}
  \def\e{\epsilon}

\def\dref#1{(\ref{#1})}

 \def\dfrac{\displaystyle\frac}
\def\be{\begin{equation}}
\def\bel{\begin{equation}\label}
\def\ee{\end{equation}}
\def\ba{\begin{array}}
\def\ea{\end{array}}
\def\banl{\begin{eqnarray}\label}
\def\ean{\end{eqnarray}}

\usepackage{pgfplots}

\newtheorem{remark}{Remark}
[section]
\newtheorem{proposition}{Proposition}[section]
\newtheorem{corollary}{Corollary}[section]
\newtheorem{lemma}{Lemma}
[section]
\newtheorem{expl}{Example}
[section]

\usepackage{algorithmic}

\usepackage{graphicx,epstopdf}

\startlocaldefs
\numberwithin{equation}{section}
\theoremstyle{plain}
\newtheorem{thm}{Theorem}[section]
\endlocaldefs

\begin{document}

\begin{frontmatter}
\title{Asymptotic Behavior of  Least Squares Estimator for  Nonlinear Autoregressive Models}
\runtitle{Asymptotic Behavior of  Least Squares Estimator}

\begin{aug}
\author{\fnms{Zhaobo} \snm{Liu}\thanksref{t2,m1}\ead[label=e1]{Liuzhaobo15@mails.ucas.ac.cn}},
\author{\fnms{Chanying} \snm{Li}\thanksref{t2,m1}\ead[label=e2]{cyli@amss.ac.cn}}

\thankstext{t2}{This work was supported in part by the National Natural Science Foundation of
China under grants 61422308 and 11688101.}
\runauthor{Z. B. LIU AND C. LI}

\affiliation{Academy of Mathematics and Systems Science, Chinese
Academy of Sciences,  and  University of Chinese Academy of Sciences\thanksmark{m1}}

\address{Key Laboratory of Systems and Control\\
Academy of Mathematics and Systems Science\\
Chinese Academy of Sciences\\
Beijing 100190\\
People's Republic of China\\
and School of Mathematical Sciences\\
University of Chinese Academy of Sciences\\
Beijing 100049\\
People's Republic of China\\
\printead{e1}\\
\phantom{E-mail:\ }\printead*{e2}}

\end{aug}

\begin{abstract}
This paper is concerned with the least squares estimator for  a basic class of nonlinear autoregressive models, whose outputs are not necessarily to be ergodic. Several asymptotic properties of the least squares estimator  have been established under mild conditions. These properties suggest
 the strong consistency of the least squares estimates in  nonlinear autoregressive models which are not divergent.

\end{abstract}

\begin{keyword}[class=MSC]
\kwd[Primary ]{62F12}
\kwd{62M10}
\kwd[; secondary ]{93E24}
\end{keyword}

\begin{keyword}
\kwd{nonlinear autoregressive models}
\kwd{ least squares}
\kwd{strong consistency}
\kwd{Harris recurrent}
\end{keyword}

\end{frontmatter}

\section{Introduction}
When it comes to estimating   nonlinear autoregressive (AR) models,
a typical case in the literature  is that  the  underlying series are ergodic. Based on this assumption, a series of  asymptotic theory has been established accordingly  (see \cite{chan93},\cite{chan98},\cite{lid12},\cite{zwx10}).
However, this good property is not always true. For example, we consider
\begin{eqnarray}\label{sys}
 y_{t+1}=\theta^{\tau}\phi(y_{t},\ldots,y_{t-n+1})+w_{t+1},~~~~~t\geq0,
\end{eqnarray}
   where $\theta$   is the $m\times1$ unknown parameter vector, $y_t, w_t$   are the scalar  observations and random   noise signals, respectively. Moreover, $\phi : \mathbb{R}^n\to \mathbb{R}^m$ is a known Lebesgue measurable vector function.
   No doubt most functions $\phi$ produce  non-ergodic sequences $\{y_t\}$.
So, this article is intended to  identify parameter $\theta$ in model \dref{sys},  whose outputs are not necessarily to be ergodic.

 It is well known that the least squares (LS) estimator is one of the most efficient algorithm in parameter estimation and
its strong consistency  for  model \dref{sys} depends crucially on the minimal eigenvalue $\lambda_{\min}(t+1)$  of  matrix 
$$P_{t+1}^{-1}=I_{m}+\sum_{i=0}^{t}\phi(y_{t},\ldots,y_{t-n+1})\phi(y_{t},\ldots,y_{t-n+1}).$$
Specifically,
in the Bayesian framework, 
\cite{ef63} and \cite{St77} showed
\begin{eqnarray}\label{77}
\left\lbrace \lim_{t\rightarrow+\infty}\lambda_{\min}(t+1)=+\infty\right\rbrace=\left\lbrace\lim_{t\rightarrow+\infty}\hat{\theta}_t=\theta\right\rbrace,
\end{eqnarray}
while \cite[Theorem 1]{laiwei82} and \cite[Lemma 3.1]{guo95} found that in the non-Bayesian framework, where $\{w_t\}$ is an approperiate martingale difference sequence,
\begin{equation}\label{jd}
\|\hat{\theta}_{t+1}-\theta\|^2=O\left(\frac{\log\left(\lambda_{\max}(t+1)\right)}{\lambda_{\min}(t+1)}\right),\quad \mbox{a.s.},
\end{equation}
where  $\lambda_{\max}(t+1)$ denotes the  maximal eigenvalue of $P_{t+1}^{-1}$.  Moreover, \cite{laiwei82} pointed out that
\begin{eqnarray}\label{lambdaminmax}
\log\left(\lambda_{\max}(t+1)\right)=o(\lambda_{\min}(t+1))
\end{eqnarray}
is  in some sense the weakest condition for the strong consistency of $\hat{\theta}_{t}$  in the non-Bayesian framework.

The eigenvalues of $P_{t+1}^{-1}$ depend on outputs $\{y_t\}$, which are produced by the nonlinear random system \dref{sys} automatically. So, checking $\lim_{t\rightarrow+\infty}\lambda_{\min}(t+1)=+\infty$ or
\dref{lambdaminmax} is not trivial  in general. But for the  linear AR model
\begin{eqnarray}\label{ARsys}
 y_{t+1}=\sum_{i=1}^n\theta_iy_{t-i+1}+w_{t+1},~~~~~t\geq0,
\end{eqnarray}
which is a special case of \dref{sys},  \cite{laiwei83} successfully  verified
\begin{eqnarray}\label{rit}
\liminf_{t\rightarrow+\infty}t^{-1}\lambda_{\min}(t+1)>0,\quad \mbox{a.s.}
\end{eqnarray}
and then completely  solved the strong consistency of the LS estimator for this basic situation. The verification of \dref{rit} in  \cite{laiwei83}, to some extent, attributes to the linear structure of model  \dref{ARsys}. As to nonlinear model \dref{sys},  we naturally  wonder if the LS estimator still  has the similar asymptotic  behavior.

In the next section, we shall establish the asymptotic properties of the LS estimator for model  \dref{sys}. By assuming some
mild conditions on $\phi$, the minimal eigenvalue of  $P_{t+1}^{-1}$ is estimated in both the Bayesian framework and non-Bayesian framework. We find that the LS estimates converge to the true parameter almost surely on the set where
vector
$(y_{t},\ldots,y_{t-n+1})^\tau$ does not diverge to infinity. Since most real system is not divergent, this means the LS estimator is very likely to be strong consistency when applied to model \dref{sys}  in practice. The proof of  the main results   is included in Section \ref{gsta}.
\section{Main Results}\label{MR}
We first consider a simplified version of model \dref{sys}  by  restricting $\phi$  as
  \begin{eqnarray}\label{sysphi}
   \phi(z_1,\ldots,z_n)=\mbox{col}\lbrace\phi^{(1)}(z_1),\ldots,\phi^{(n)}(z_n)\rbrace,
   \end{eqnarray}
   where  $\phi^{(i)}=(f_{i1},\ldots,f_{im_i})^\tau : \mathbb{R}\to \mathbb{R}^{m_i},  i=1,\ldots,n$ are some  known Lebesgue measurable vector functions and $m_i\geq 1$ are $n$ integers satisfying $\sum_{i=1}^{n}m_i=m$.  Without loss of generality, let $y_t=0$ for $t<0$.
We discuss the parameter estimation of model \dref{sys} and \dref{sysphi} by two cases.      In  Subsection \ref{bframe}, parameter $\theta$ is treated as a random variable, while it is a fixed vector in Subsection \ref{ConstantP}.

Next, we establish the asymptotic theory of the LS estimator for the general AR model \dref{sys} in Subsection \ref{extension}.

\subsection{\textbf{Bayesian Framework}}\label{bframe}
Consider model \dref{sys} and \dref{sysphi}. Assume

\begin{description}

\item[A1]
The noise $\{w_t\} $ is an  i.i.d random sequence with  $w_1 \sim N(0,1)$ and  parameter $ \theta\sim N(\theta_0,I_{m})$ is independent of $\{w_t\}$.

\item[A2]  There are some open sets $\lbrace E_i\rbrace_{i=1}^{n}$ belonging to  $\mathbb{R}$ such that \\
(i) $f_{ij}\in C(\mathbb{R})$ and $f_{ij} \in C^{m_i}(E_i)$, $1\leq j\leq m_i$, $1\leq i\leq n$;\\
(ii) for every unit vector $x\in\mathbb{R}^{m}$, there is a point $y\in \prod\nolimits_{i=1}^{n} E_i$ such that
$ |\phi^{\tau}(y)x| \neq 0.   $
\end{description}

\begin{remark}
By Assumption A2(ii), for every unit vector $x\in\mathbb{R}^{m}$,
\begin{eqnarray}
\ell\left(\left\lbrace y\in \prod\nolimits_{i=1}^{n} E_i: |\phi^{\tau}(y)x|>0 \right\rbrace\right)>0,\nonumber
\end{eqnarray}
where $\ell$ denotes the Lebesgue measure.
\end{remark}
When $n=1$, Assumption A2 can be relaxed as
\begin{description}
\item[A2']  $f_{1i} \in C^{m_1}(E_1), i=1,\ldots,m_1$ are linearly independent in $E_1$, and  $\phi$ is bounded in every compact set.
\end{description}
The LS estimate $\hat{\theta}_t$ for parameter $\theta$ can be recursively defined by
\begin{eqnarray}\label{LS2}
\left\{
\begin{array}{l}
\hat{\theta}_{t+1}=\hat{\theta}_t+P_{t+1}\phi_t(y_{t+1}-\phi_t^\tau\hat{\theta}_t)\\
P_{t+1}=P_t-(1+\phi_t^\tau P_t\phi_t)^{-1}P_t\phi_t\phi_t^\tau P_t,~~P_0=I_{m}\\
\phi_t=\phi(y_{t},\ldots,y_{t-n+1}),~~t\geq0
\end{array},
\right.
\end{eqnarray}
where $\hat{\theta}_0$ is the deterministic initial condition of the algorithm and $\phi_0$ is the random initial vector of  system \dref{sys}. Clearly, by \dref{sys} and \dref{LS2},
\begin{eqnarray}\label{yt}
P_{t+1}^{-1}=I_{m}+\sum_{i=0}^{t}\phi_i\phi
_i^{\tau}.
\end{eqnarray}
We provide a simple way to estimate the minimal eigenvalue of $P_{t+1}^{-1}$, which is denoted as $\lambda_{\min}(t+1)$.  Let
\begin{eqnarray}\label{NtM}
N_t(M)\triangleq \sum_{i=1}^{t}I_{\lbrace \|Y_{i}\|\leq M\rbrace},
\end{eqnarray}
where $Y_t\triangleq (y_{t+n-1},\ldots,y_{t})^{\tau}$ and  $ M>0$ is a constant. Then, in terms of  $N_t(M)$, our estimate of $\lambda_{\min}(t+1)$ is readily available  by

 \begin{thm}\label{tzz}
Under Assumptions A1--A2, for any constant $M>0$,
\begin{eqnarray}
\liminf_{t\rightarrow+\infty}\frac{\lambda_{\min}(t+1)}{N_t(M)}>0\quad \mbox{a.s.}~\mbox{on}~ \Omega(M),
\end{eqnarray}
where $\Omega(M)\triangleq\left\lbrace \lim_{t\rightarrow+\infty}N_t(M)=+\infty\right\rbrace. $
\end{thm}

\begin{corollary}\label{a12}
Let Assumptions A1--A2 hold. Then,
\begin{eqnarray}\label{strconsis}
\lim_{t\rightarrow+\infty}\hat{\theta}_t=\theta\quad \mbox{a.s.}~\mbox{on}~\left\lbrace \liminf_{t\rightarrow+\infty}\|Y_{t}\|<+\infty\right\rbrace
\end{eqnarray}
\end{corollary}
\begin{remark}
If Assumption A2(ii)  fails, then
 $$\ell(\lbrace y\in \mathbb{R}^n: |\phi^{\tau}(y)x|>0\rbrace)=0$$ for some unit vector $x\in\mathbb{R}^{m}$. Therefore, by \dref{yt}, as $t\rightarrow\infty$,
$$\lambda_{\min}(t+1)=O(1),   \quad  \mbox{a.s.}.$$
In view of \dref{77},   $\hat{\theta}_t$ cannot converge to the true parameter $\theta$. So,   Assumption A2(ii) is necessary for the strong consistency of the LS estimates $\lbrace \hat{\theta}_t\rbrace_{t\geq 0}$.
\end{remark}

\subsection{\textbf{Constant Parameter}}\label{ConstantP}
Consider model \dref{sys} and \dref{sysphi}, where $\theta$ is a non-random parameter.   Assume
\begin{description}
\item[A1']
$\{w_t\} $ is an  i.i.d random sequence with $E w_1=0$ and $E|w_1|^{\beta}<+\infty$ for some $\beta>2$. Moreover,
$w_1$ has
 a density  $\rho(x)$  such that
 for every proper interval $I\subset\mathbb{R}$,
 $$
 \inf_{x\in I}\rho(x)>0
\quad \mbox{and}\quad  \sup_{x\in \mathbb{R}}\rho(x)<+\infty.$$
\end{description}
In this case,  the LS estimator is constructed from partial data. More specifically,  for some constant $C_{\phi}>0$,   $\phi_t$ in \dref{LS2} is modified as
 $$\phi_t\triangleq I_{\lbrace \|Y_{t-n+1}\|\leq C_{\phi}\rbrace}\phi(y_{t},\ldots,y_{t-n+1}).$$
Let $\lambda_{\min}(t+1)$ and $\lambda_{\max}(t+1)$ denote the minimal and maximal eigenvalues of $P_{t+1}^{-1}$ in \dref{yt}.
Define
$r_t\triangleq \sum_{i=0}^{t}\|\phi_i\|^2+1$ as the trace of $P_{t+1}^{-1}$.  Note that  $\frac{r_t}{ \lambda_{\max}(t+1) }\in [1,n] $ and $r_{t}=O(N_t(C_{\phi}))$, where $N_t(\cdot)$ is defined by \dref{NtM}. Then,
an analogous version of Theorem \ref{tzz} is deduced as follows:

\begin{thm}\label{tzz1}
Under Assumptions A1' and A2, there is a constant $M_{\phi}>0$  depending only on $\phi$ such that for any  $C_{\phi}>M_{\phi}$  and  $M>0$,
\begin{eqnarray*}
\liminf_{t\rightarrow+\infty}\frac{\lambda_{\min}(t+1)}{N_t(M)}>0
\quad \mbox{a.s.}~\mbox{on}~\Omega(M).
\end{eqnarray*}
Furthermore, if $M\geq C_{\phi}$, then $\|\hat{\theta}_t-\theta\|^2=O(\frac{\log N_t(M) }{N_t(M)})$     a.s. on set $\Omega(M)$.
\end{thm}

\begin{remark}
Theorem \ref{tzz1} indicates that  \dref{strconsis} holds  under Assumptions A1' and A2. In most practical situations,
\begin{eqnarray}\label{Hrecurrent}
P\left\lbrace \liminf_{t\rightarrow+\infty}\|Y_{t}\|<+\infty\right\rbrace=1
\end{eqnarray}
and the strong consistency of the LS estimates is thus guaranteed. Note that Assumption A1' and \dref{Hrecurrent} imply that $\{y_t\}_{t\geq 1}$  in model \dref{sys}  is in fact an aperiodic Harris recurrent Markov chain and hence admits an invariant measure. Some  integrability  assumptions on the invariant measure might also lead to the consistency of the LS estimates (e.g.\cite{lid16}). However, it is  not yet clear  that  the  invariant measure  of  such  a   nonlinear autoregressive model        ever has the desired properties for estimation. 

\end{remark}

\begin{expl}
Consider a parametric autoregressive model of the form:
\begin{eqnarray}\label{tar}
y_{t+1}=\sum_{j=1}^{n}\theta_j g(y_{t})I_{\lbrace y_{t}\in D_j\rbrace}+y_{t}I_{\lbrace y_{t}\in D_{n+1}\rbrace}+w_{t+1},\quad y_0=0,
\end{eqnarray}
where
$g(\cdot)$ is bounded in any compact set, $\lbrace D_j\rbrace_{j=1}^{n}$ are some compact subsets of $\mathbb{R}$ with positive Lebesgue measure and $D_{n+1}=(\bigcup_{j=1}^{n}D_i)^c$. Let noises $\lbrace w_t\rbrace_{t\geq 1} $ satisfy Assumption A1' and  unknown parameters $\theta_1,\ldots,\theta_n\in\mathbb{R}$. Considering the properties of random walks, $\lbrace y_{t}\rbrace_{t\geq 1}$ must fall into $\bigcup_{j=1}^{n}D_i$ infinitely many times. Then,  it follows that $\lbrace y_{t}\rbrace_{t\geq 1}$ fulfills \dref{Hrecurrent}. Hence  Theorems \ref{tzz} and \ref{tzz1} can be applied and the strong consistency of the LS estimates is established. If $g(x)=x$,  model \dref{tar} turns out to be the familiar   threshold autoregressive  (TAR) model.
\end{expl}

\subsection{\textbf{Asymptotic Theory for General Model}}\label{extension}
Let us return to  model \dref{sys} and rewrite
$$\phi(z)=\mbox{col}\lbrace f_1(z),\ldots,f_m(z)\rbrace,$$
   where $z=(z_1,\ldots,z_n)^{\tau}$ and $f_i : \mathbb{R}^n\to \mathbb{R},  i=1,\ldots,n$ are some  known Lebesgue measurable vector functions.
A natural question in this part is  whether the asymptotic behavior of the LS estimator in Theorems \ref{tzz} and \ref{tzz1} still holds for model \dref{sys}? To  this end, assume

\begin{description}
\item[A3] There is a bounded open set $E\subset \mathbb{R}^{n}$ and a number $\delta^{*}>0$ such that\\
(i) $f_i\in C(\mathbb{R}^n)$, $1\leq i\leq m$;\\
(ii) for every unit vector $x\in\mathbb{R}^n$,
\begin{eqnarray}\label{Jordan}
J\left(\lbrace y\in \overline{E}: |\phi^{\tau}(y)x|=\delta^{*}\rbrace \right)=0,
\end{eqnarray}
where $J(\cdot)$ denotes the Jordan measure. In addition,
\begin{eqnarray}\label{dl*}
\inf_{\|x\|=1}\ell\left(\lbrace y\in E: |\phi^{\tau}(y)x|>\delta^{*}\rbrace \right)>0.
\end{eqnarray}

\end{description}


With the proof  placed in Appendix \ref{AppB}, our problem is addressed by

\begin{thm}\label{tzz2}
Theorems \ref{tzz} and \ref{tzz1} hold for model \dref{sys} if Assumption A2 is replaced by A3.
\end{thm}

\begin{expl}
Consider the following exponential autoregressive  model (EXAR) with noises $\lbrace w_{t}\rbrace_{t\geq 1}$ satisfying A1':
\begin{eqnarray}\label{EXAR}
y_{t+1}=\sum_{j=1}^{n}(\alpha_j+\beta_j e^{-\gamma y_{t}^2})y_{t-j+1}+w_{t+1},
\end{eqnarray}
where $\gamma$ is known and $\a_j,\b_j, j=1,2,\ldots, n$ are unknown parameters. It can be checked that   Assumption A3 holds for model \dref{EXAR}. Furthermore, in most practical cases,  outputs $\lbrace y_{t}\rbrace_{t\geq 1}$  produced by  the above EXAR models   fulfill  \dref{Hrecurrent}. So,  the LS estimator is often effective for model \dref{EXAR} due to Theorem \ref{tzz2}.
\end{expl}

\section{Proofs of Theorems \ref{tzz} and \ref{tzz1}}\label{gsta}
It is obvious that to show Theorems \ref{tzz} and \ref{tzz1}, it suffices to prove
\begin{proposition}\label{cruc}
Under Assumptions  A1' and A2,  let $\theta$ be a random variable  independent of $\lbrace w_t\rbrace_{t\geq 1}$.
Then, there is a constant $M_{\phi}>0$  depending only on $\phi$ such that for any  $C_{\phi}>M_{\phi}$  and  $M,K>0$,
\begin{eqnarray}\label{eigenN}
\liminf_{t\rightarrow+\infty}\frac{\lambda_{\min}(t+1)}{N_t(M)}>0\quad \mbox{a.s. on } \Omega(M)\cap \lbrace \|\theta\|\leq K\rbrace.
\end{eqnarray}
\end{proposition}

Borrowing the idea of \cite{lilam13},  the proof of Proposition \ref{cruc} will be completed  in the following three subsections.\\
\textit{Section \ref{cu}}: Observe that
\begin{eqnarray}
\lambda_{\min}(t+1)&=&\min_{\|x\|=1}x^{\tau}\left(I_{m}+\sum_{i=1}^{t}\phi_i\phi_i^{\tau}\right)x\nonumber\\
&=&1+\min_{\|x\|=1}\sum_{i=1}^{t}(\phi_i^{\tau}x)^2,\nonumber
\end{eqnarray}
so for any unit vector $x\in\mathbb{R}^{m}$, we shall construct a set $U_x\subset \overline{B(0,C_{\phi})}\subset\mathbb{R}^n$ such that $\inf_{y\in U_x}|\phi^{\tau}(y)x|\geq \delta$ for some $\delta>0$.\\
\textit{Section \ref{32}}: We shall analyze  the properties of $U_x$ and derive a key technique result for our problem in Lemma \ref{xy}.\\
\textit{Section \ref{33}}: This section is intended to prove \dref{eigenN} by estimating the  frequency of $\lbrace Y_t\rbrace _{t\geq 1}$ falling into $U_x$.

\subsection{\textbf{Construction of $U_x$}}\label{cu}
The important set $U_x$ is constructed from a finite family of disjoint open intervals $\{S_{i}^{j}(q)\}$ defined below.
\subsubsection{ Open Intervals  $S_{i}^{j}(q)$}
We  claim that
for each $i\in[1,n]$, there exists a finite family of disjoint open intervals $\lbrace S_{i}^{j}(q)\rbrace_{j=1}^{p_i}$  for some $q\in\mathbb{N}^{+}$  fulfilling
:\\
(i) $\phi^{(i)}\in C^{m_i}$ in $\bigcup_{j=1}^{p_i}S_{i}^{j}(q)$;\\
(ii) $\bigcup_{j=1}^{p_i}S_{i}^{j}(q)$ has no points in  $Z_{s}^{2}(i)$ defined later in \dref{z123};\\
(iii) For every unit vector $x\in\mathbb{R}^{m}$,
\begin{eqnarray}
\ell\left(\left\lbrace  y\in \prod\nolimits_{i=1}^{n}\bigcup\nolimits_{j=1}^{p_i}S_{i}^{j}(q): |\phi^{\tau}(y)x|>0\right\rbrace\right)>0.
\end{eqnarray}
We preface the proof of the claim with  several auxiliary lemmas.
\begin{lemma}\label{useful}
Let $\lbrace U_j\rbrace_{j\geq 1}$ be a sequence of open sets in $\prod_{i=1}^{n} E_i$ satisfying $U_1\subset U_2\ldots\subset U_j\subset\ldots$ and
\begin{equation}\label{uj}
\lim_{j\rightarrow+\infty}U_{j}=U,
\end{equation}
where $U$ is a non-empty open set that
\begin{eqnarray*}
\ell(\lbrace y\in U : |\phi^{\tau}(y)x|>0\rbrace)>0,\quad \forall x\in \mathbb{R}^{m},~\|x\|=1.
\end{eqnarray*}
Then, there is an integer $j$ such that
\begin{eqnarray*}
\ell(\lbrace y\in U_j : |\phi^{\tau}(y)x|>0 \rbrace)>0,\quad \forall x\in \mathbb{R}^{m},~\|x\|=1.
\end{eqnarray*}
\end{lemma}
\begin{proof}
If the assertion is not true, then by the continuity of $\phi$ in  Assumption A2(i),  for each $j\geq 1$, there is a vector $x^{j}\in\mathbb{R}^{m}$ with $\|x^{j}\|=1$ such that
\begin{eqnarray}\label{p1}
\phi^{\tau}(y)x^{j}=0,\quad \forall y\in U_j.
\end{eqnarray}
It follows that there is a subsequence $\lbrace x^{n_{i}}\rbrace_{i\geq 1}$ of $\lbrace x^{j}\rbrace_{j\geq 1}$ satisfying
\begin{eqnarray}\label{p2}
\lim_{i\rightarrow+\infty}x^{n_{i}}=x^{\infty},
\end{eqnarray}
where $\|x^{\infty}\|=1$. On the other hand,
\begin{eqnarray*}
\ell(\lbrace  y\in U: |\phi^{\tau}(y)x^{\infty}|>0\rbrace)>0,
\end{eqnarray*}
so there is a $y^{*}\in U$ such that
\begin{equation}\label{p3}
|\phi^{\tau}(y^{*})x^{\infty}|>0.
\end{equation}
By \dref{uj}, there is an integer $m'\geq 1$ such that $y^{*}\in U_{j}$ for all $j\geq m'$, and hence \dref{p1}--\dref{p3} yield
\begin{equation*}
0<|\phi^{\tau}(y^{*})x^{\infty}|=\lim_{i\rightarrow+\infty}|\phi^{\tau}(y^{*})x^{n_i}|=0,
\end{equation*}
which leads to a contradiction.
\end{proof}

\begin{remark}\label{Eifinite}
Since every open $E_i\subset \mathbb{R}, i\in[1,n]$ is a countable union of disjoint open intervals,
 Lemma \ref{useful} implies that
  there is an open set $E_i'\subset E_i$ such that $E_i'$ consists of a finite number of disjoint open intervals and
\begin{eqnarray*}
\ell(\lbrace y\in \prod_{i=1}^n E_i': |\phi^{\tau}(y)x|>0\rbrace)>0,\quad \forall x\in \mathbb{R}^{m},~\|x\|=1.
\end{eqnarray*}
So, without loss of generality, assume  each $E_i$  in the sequel is a finite union of    disjoint open intervals.
\end{remark}

Now, we introduce a series of operators. Denote $D$ as the differential operator, then for any sufficiently smooth
 functions $\lbrace g_l\rbrace_{l\geq 1}$, recursively define
\begin{eqnarray}\label{Lambdal}
\left\{
\begin{array}{l}
\Lambda_1(g_1)\triangleq{g_1}\\
\Lambda_{l+1}(g_1,\cdots,g_{l+1})\triangleq{\Lambda_l\left(\frac{D g_1}{D g_{l+1}},\cdots,\frac{Dg_l}{Dg_{l+1}}\right)},\quad l\geq 1
\end{array}.
\right.
\end{eqnarray}
These operators $\lbrace \Lambda_l\rbrace_{l\geq 1}$  have the following property:

\begin{lemma}\label{Lambda}
Let  functions $\lbrace g_i\rbrace_{i=1}^{l+1}$, $l\in \mathbb{N}^+$ be   sufficiently smooth, then
\begin{eqnarray}\label{amwy}
\Lambda_{l+1}(g_1,\ldots,g_{l+1})=\frac{D(\Lambda_{l}(g_1,g_3,\ldots,g_{l+1}))}{D(\Lambda_{l}(g_2,g_3,\ldots,g_{l+1}))}.
\end{eqnarray}
\end{lemma}
\begin{proof}
We use the induction method to show this lemma. By the definition of $\Lambda_2$, it is easy to check
\begin{eqnarray}
\Lambda_{2}(g_1,g_2)=\Lambda_1\left(\dfrac{Dg_1}{Dg_2}\right)=\frac{Dg_1}{Dg_2}=\frac{D(\Lambda_{1}(g_1))}{D(\Lambda_1(g_2))}.\nonumber
\end{eqnarray}
Let $k\geq 2$.  Suppose  \dref{amwy} holds for any  functions $\lbrace g_i\rbrace_{i=1}^{l+1}$, $l=k-1$, then
\begin{eqnarray}
&&\Lambda_{k}\left(\frac{Dg_1}{Dg_{k+1}},\ldots,\frac{Dg_k}{Dg_{k+1}}\right)\nonumber\\
&=&\frac{D\left(\Lambda_{k-1}\left(\frac{Dg_1}{Dg_{k+1}},\frac{Dg_3}{Dg_{k+1}},\ldots,\frac{Dg_k}{Dg_{k+1}}\right)\right)}{D\left(\Lambda_{k-1}\left(\frac{Dg_2}{Dg_{k+1}},\frac{Dg_3}{Dg_{k+1}},\ldots,\frac{Dg_k}{Dg_{k+1}}\right)\right)},\nonumber
\end{eqnarray}
and hence by \dref{Lambdal},
\begin{eqnarray}
\Lambda_{k+1}(g_1,g_2,\ldots,g_{k+1})&=&\Lambda_{k}\left(\frac{Dg_1}{Dg_{k+1}},\ldots,\frac{Dg_k}{Dg_{k+1}}\right)\nonumber\\
&=& \frac{D\left(\Lambda_{k-1}\left(\frac{Dg_1}{Dg_{k+1}},\frac{Dg_3}{Dg_{k+1}},\ldots,\frac{Dg_k}{Dg_{k+1}}\right)\right)}{D\left(\Lambda_{k-1}\left(\frac{Dg_2}{Dg_{k+1}},\frac{Dg_3}{Dg_{k+1}},\ldots,\frac{Dg_k}{Dg_{k+1}}\right)\right)}\nonumber\\
&=& \frac{D\left(\Lambda_{k}\left(g_1,g_3,\ldots,g_{k+1}\right)\right)}{D\left(\Lambda_{k}\left(g_2,g_3,\ldots,g_{k+1}\right)\right)},\nonumber
\end{eqnarray}
which completes the induction.
\end{proof}

Before proceeding to the next lemma, we define some notations. Let $l_1<\cdots<l_s$ be
  $s$ positive integers.  For each $k\in [1, s]$,  denote $\mathcal{H}_k^{(l_1,\ldots,l_s)}$ as the $k$-permutations of $\lbrace l_1,\ldots,l_s\rbrace$. That is,
\begin{eqnarray}
\mathcal{H}_k^{(l_1,\ldots,l_s)}&\triangleq &{\lbrace(i_1,\ldots,i_k): i_j\in \lbrace l_1,\ldots,l_s\rbrace,1\leq j\leq k;i_r\neq i_j~\mbox{if}~r\neq j\rbrace}.\nonumber
\end{eqnarray}
Now, let $i\in[1,n]$. For each $(i_1,\ldots,i_k)\in \mathcal{H}_k^{(1,\ldots,m_i)}$, $k\in[1,m_i]$, define
\begin{eqnarray}\label{GammaL}
\Gamma_{(i_1,\ldots,i_k)}^{(i)}\triangleq{\Lambda_k\left(f_{ii_1},\ldots,f_{ii_k}\right)},\quad \bar\Gamma_{s}^{(i)}\triangleq{D\Gamma_{s}^{(i)}},
\end{eqnarray}
 and for any $s\in \mathcal{H}_i\triangleq{ \bigcup_{k=1}^{m_i} \mathcal{H}_k^{(1,\ldots,m_i)}}$,
\begin{equation*}
W_s(i)\triangleq\left\lbrace y: \bar\Gamma_{s}^{(i)}(y)~\mbox{is well-defined}\right\rbrace.
\end{equation*}
Given function $g$, denote $A(g)\triangleq{\lbrace x: g(x)=0\rbrace}$. In addition, for any two sets $\mathcal{X}_1,\mathcal{X}_2\subset \mathbb{R}$, we say that $\mathcal{X}_1$ is \emph{locally dense} in $\mathcal{X}_2$, if
$\mathcal{X}_1$ is not nowhere dense in $\mathcal{X}_2$. That is, there exists a nonempty open interval $\mathcal{X}_3\subset\mathcal{X}_2$ such that $\mathcal{X}_3\subset\overline{\mathcal{X}_1}$.
With the above pre-definitions, we assert

\begin{lemma}\label{ws}
Let integers $i\in[1,n]$, $k\in [2,m_i]$ and array $s^{*}\in\mathcal{H}_k^{(1,\ldots,m_i)}$. Under Assumption A2,   there is a set $H_{ik}\subset \bigcup_{j<k}{H}_j^{(1,\ldots,m_i)}$ such that
\begin{eqnarray}\label{wsc}
W_{s^{*}}^{c}(i)\cap E_i=\bigcup\nolimits_{s\in H_{ik}}(A(\bar\Gamma_{s}^{(i)})\cap E_i).
\end{eqnarray}
Moreover, let $U\subset E_i$ be a non-empty set with
\begin{eqnarray}\label{Uden}
U\subset \overline{W_{s^{*}}^{c}(i)}\quad \mbox{and}\quad \mbox{int}(W_{s^{*}}^{c}(i)\cap U)=\emptyset,
\end{eqnarray}
then we can find some  $j<k$ and $s'\in\mathcal{H}_j^{(1,\ldots,m_i)}$ such that $A(\bar\Gamma_{s'}^{(i)})$ is locally dense in $U$ and $\mbox{int}(A(\bar\Gamma_{s'}^{(i)})\cap U)=\emptyset$.
\end{lemma}
\begin{proof}
We first prove \dref{wsc}  for the given $i$ and $k$. Let $s_{k,1}=s^{*}$, then for each $j=k,\ldots,2$, Lemma \ref{Lambda} and \dref{GammaL} indicate that there exist some indices
$s_{j-1,1}, s_{j-1,2}\in\mathcal{H}_i$  such that
\begin{eqnarray}\label{gam}
\Gamma_{s_{j,1}}^{(i)}=\frac{\bar\Gamma_{s_{j-1,1}}^{(i)}}{\bar\Gamma_{s_{j-1,2}}^{(i)}}.
\end{eqnarray}
Denote $H_{ik}\triangleq   \{s_{j,2}, j=1, \ldots,k-1 \} $.

Note that by  \dref{Lambdal}, \dref{GammaL} and Assumption A2(i), it is easy to see
$$\lbrace y\in E_i: \Gamma_{s^{*}}^{(i)}(y)~\mbox{is well-defined}\rbrace=\lbrace y\in E_{i}: D\Gamma_{s^{*}}^{(i)}(y)~\mbox{is well-defined}\rbrace.$$
In addition,  Lemma \ref{Lambda} infers that for each $j=1,\ldots,k$,
$$\lbrace y\in E_i: \Gamma_{s_{j-1,1}}^{(i)}(y)~\mbox{is well-defined}\rbrace=\lbrace y\in E_i: \Gamma_{s_{j-1,2}}^{(i)}(y)~\mbox{is well-defined}\rbrace.$$
Then, by \dref{gam},
\begin{eqnarray}
W_{s^{*}}^{c}(i)\cap E_i&= &\lbrace y\in E_i: \Gamma_{s^{*}}^{(i)}(y)~\mbox{is undefined}\rbrace\nonumber\\
&= &\lbrace y\in E_i: \bar\Gamma_{s_{k-1,1}}^{(i)}(y)~\mbox{is undefined}\rbrace\cup A(\bar\Gamma_{s_{k-1,2}}^{(i)})\nonumber\\
& = &\cdots=\bigcup\nolimits_{s\in H_{ik}}(A(\bar\Gamma_{s}^{(i)})\cap E_i),\nonumber
\end{eqnarray}
which is exactly \dref{wsc}.
So, if \dref{Uden} holds,  for every $s\in H_{ik}$,
$\mbox{int}(A(\bar\Gamma_{s}^{(i)})\cap U)=\emptyset.$ Finally, we show that for some $s'\in H_{ik}$,  $A(\bar\Gamma_{s'}^{(i)})$ is locally dense in $U$. Otherwise, $A(\bar\Gamma_{s}^{(i)})$ is nowhere dense in $U$  for every $s\in  H_{ik}$. 
This means there are  a series of nonempty open intervals $U_{1}\subset\cdots\subset U_{k-1}\subset U$ such that
$$U_{j}\cap \overline{ A(\bar\Gamma_{s_{l,2}}^{(i)})}=\emptyset \quad \mbox{for all } l=j,\ldots, k-1.$$
As a consequence,  by \dref{wsc},
\begin{eqnarray*}
U_{1}  \cap \overline{W_{s^{*}}^{c}(i)}=
U_{1}  \cap \left(\bigcup\nolimits_{j=1}^{k-1} \overline{ A(\bar\Gamma_{s_{j,2}}^{(i)})}\right)=\emptyset,
\end{eqnarray*}
which contradicts to \dref{Uden} due to $U_{1}\subset U$.
\end{proof}

Now,  we are ready to construct $\lbrace S_{i}^{j}(q)\rbrace_{j=1}^{p_i}$. For this, we classify the sets $\overline{A(\bar\Gamma_{s}^{(i)})}, s\in \mathcal{H}_i, i\in[1,n]$ into three types:
\begin{eqnarray}\label{z123}
\left\{
\begin{array}{l}
Z_{s}^{1}(i)=\mbox{int}( A(\bar\Gamma_{s}^{(i)}))\\
 Z_{s}^{2}(i)=d(A(\bar\Gamma_{s}^{(i)}))\backslash Z_{s}^{1}(i)\\
 Z_{s}^3(i)= A(\bar\Gamma_{s}^{(i)})\backslash d(A(\bar\Gamma_{s}^{(i)}))
\end{array},\quad i\in[1,n],
\right.
\end{eqnarray}
where $d(A)$ denotes the derived set of $A$.
Observe that $Z_{s}^{1}(i)$ can be expressed by a  countable union of disjoint open intervals and  $Z_{s}^3(i)$  is in fact the set of the isolated points of $A(\bar\Gamma_{s}^{(i)})$. Both the two sets have good topological properties. However, the structure of $Z_s^2(i)$ is not that clear. 
 Therefore, we define the following sets to exclude $Z_{s}^2(i)$:
\begin{equation*}
S(i)\triangleq{E_i\Big\backslash\left(\bigcup\nolimits_{s\in\mathcal{H}_i}Z_{s}^{2}(i)\right)},\quad i\in[1,n],
\end{equation*}
which  are clearly some open sets.

The key idea of the construction of $\lbrace S_{i}^{j}(q)\rbrace_{j=1}^{p_i}$ is to find a proper subset of $S(i)$ for each $i\in [1,n]$. To begin with, we prove an important lemma.
\begin{lemma}\label{sw}
Under Assumption A2, for any unit vector $x\in\mathbb{R}^{m}$,
\begin{eqnarray}
\ell\left(\left\lbrace  y\in \prod_{i=1}^{n}S(i): |\phi^{\tau}(y)x|>0\right\rbrace\right)>0.
\end{eqnarray}
\end{lemma}
\begin{proof}
We show the lemma   in a way of  reduction to absurdity. Suppose there exists some $x\in\mathbb{R}^{m}$ with $\|x\|=1$ such that
\begin{eqnarray}\label{elly}
\ell\left(\left\lbrace y\in \prod_{i=1}^{n}S(i): |\phi^{\tau}(y)x|>0\right\rbrace\right)=0.
\end{eqnarray}
As $\phi(\cdot)$ is continuous on open set $\prod_{i=1}^{n}S(i)\subset\prod_{i=1}^{n} E_{i}$, then
\begin{equation}\label{phi0}
\phi^{\tau}(y)x=0,\quad \forall y\in \prod_{i=1}^{n}S(i)
\end{equation}
Note that Assumption A2(ii)  yields
\begin{eqnarray}
\ell\left(\left\lbrace y\in \prod_{i=1}^{n} E_{i}: |\phi^{\tau}(y)x|>0\right\rbrace\right)>0,\nonumber
\end{eqnarray}
which together with \dref{elly} implies
\begin{equation*}
\ell\left(\left\lbrace y\in \prod_{i=1}^{n} E_{i}\backslash\prod_{i=1}^{n}S(i): |\phi^{\tau}(y)x|>0\right\rbrace\right)>0.
\end{equation*}
Consequently, there is a $y^{*}=(y_{1}^{*},\ldots,y_{n}^{*})\in \prod_{i=1}^{n} E_{i}\backslash\prod_{i=1}^{n}S(i)$ such that $|\phi^{\tau}(y^{*})x|>0$. By the continuity of $\phi(\cdot)$ on $\prod_{i=1}^{n} E_{i}$, there is  $\varepsilon>0$ such that
\begin{eqnarray}\label{phiyx}
|\phi^{\tau}(y)x|>0,\quad \forall y\in\prod_{i=1}^{n}(y_i^{*}-\varepsilon,y_i^{*}+\varepsilon)\subset \prod_{i=1}^{n} E_{i}.
\end{eqnarray}
On account of \dref{phi0} and \dref{phiyx}, we deduce
\begin{eqnarray*}
\prod_{i=1}^{n}(y_i^{*}-\varepsilon,y_i^{*}+\varepsilon)&\subset& \prod_{i=1}^{n} E_{i}\backslash\prod_{i=1}^{n}S(i)\nonumber\\
&=&\bigcup_{i=1}^n\left(\prod_{j=1}^{i-1} E_{j} \times \left(\bigcup_{s\in \mathcal{H}_i}Z_{s}^{2}(i)\right)\times  \prod_{j=i+1}^{n} E_{j}\right),
\end{eqnarray*}
which immediately yields that for some index $i\in [1,n]$,
\begin{eqnarray}\label{yvy}
V_i\triangleq(y_i^{*}-\varepsilon,y_i^{*}+\varepsilon)\subset\bigcup\nolimits_{s\in \mathcal{H}_i}Z_{s}^{2}(i).
\end{eqnarray}

Next, we  show  \dref{yvy} is impossible. To this end, note that $Z_{s}^{2}(i)$ is closed for each $s\in \mathcal{H}_i$, and hence \dref{yvy} implies that there is an integer $k\in [1,m_i]$ and
an array $s^*\in  \mathcal{H}_k^{(1,\ldots,m_i)}$ such that  $Z_{s^*}^{2}(i)$ is locally dense in $V_i$.
Let $k$ be the smallest integer for such $s^*$.

Now, fix the above $i\in [1,n], k\in [1,m_i]$ and $s^*\in  \mathcal{H}_k^{(1,\ldots,m_i)}$.
Since
$Z_{s^*}^{2}(i)$ is locally dense in $V_i$, there is  an open interval  $V_i'\subset V_i$ such that $Z_{s^*}^{2}(i)$ is dense in $V_i'$. Moreover,  $Z_{s^*}^{2}(i)$ is closed,  so  $V_{i}'\subset Z_{s^*}^2(i)$ and thus $V_{i}^{'}\cap Z_{s^*}^1(i)=\emptyset$.  In addition, $\bar\Gamma_{s^*}^{(i)}$ is continuous in $W_{s^*}(i)\cap E_i$, by \dref{z123},  $W_{s^{*}}(i)\cap Z_{s^{*}}^2(i)\cap E_i  \subset A(\bar\Gamma_{s^{*}}^{(i)})$. Consequently,
\begin{eqnarray}\label{xrc}
\mbox{int}\left(W_{s^{*}}(i)\cap Z_{s^{*}}^2(i)\cap E_i\right)    \subset \mbox{int}(A(\bar\Gamma_{s^{*}}^{(i)})\cap E_{i})   = Z_{s^{*}}^{1}(i)\cap E_{i}.
\end{eqnarray}
Moreover, $V_i'$ is an open interval belongs to $E_i$ and $V_i'\cap (Z_{s^*}^2(i))^c=\emptyset$, then
\begin{eqnarray}\label{VinW}
V_i'=V_i'\backslash Z_{s^*}^1(i)&\subset & V_i'\backslash\mbox{int}\left(W_{s^*}(i)\cap Z_{s^*}^2(i)\right)\nonumber\\
&\subset&\overline{V_i'\backslash(W_{s^*}(i)\cap Z_{s^*}^2(i))}=\overline{V_i'\cap W_{s^*}^{c}(i)}\subset\overline{W_{s^*}^{c}(i)}.
\end{eqnarray}
Note that $\bar\Gamma_{s}^{(i)}$ are  well defined in $\mathbb{R}$ for  all $s\in  \mathcal{H}_1^{(1,\ldots,m_i)}$ by Assumption A2(i), which shows
$W_{s}^{c}(i)\cap E_i=\emptyset$. 
Then, \dref{VinW} implies $k\geq2$. Furthermore, by $Z_{s^{*}}^{2}(i)\subset\overline{A(\bar\Gamma_{s^{*}}^{(i)})}$, it yields
\begin{eqnarray}\label{IntWem}
\mbox{int}\left(W_{s^*}^{c}(i)\cap V_i'\right)&\subset &\mbox{int}\left(A^{c}(\bar\Gamma_{s^*}^{(i)})\cap V_i'\right)=\left(\overline{(A^{c}(\bar\Gamma_{s^*}^{(i)})\cap V_i')^{c}}\right)^{c}\nonumber\\
&=& \left(\overline{A(\bar\Gamma_{s^*}^{(i)})}\cup(V_i')^{c}\right)^{c}\subset \left(Z_{s^*}^{2}(i)\cup(V_i')^{c}\right)^{c}=\emptyset.
\end{eqnarray}
Applying Lemma \ref{ws}, \dref{VinW} and \dref{IntWem} indicate that
we can find some $j<k$ and $s'\in\mathcal{H}_j^{(1,\ldots,m_i)}$ such that $A(\bar\Gamma_{s'}^{(i)})$ is locally dense in $V_i'$ and $\mbox{int}(A(\bar\Gamma_{s'}^{(i)})\cap V_i')=\emptyset$. So, there is an open interval $V''_i\subset V'_i$ such that $V''_i\subset d(A(\bar\Gamma_{s'}^{(i)}))$ and  $\mbox{int}(A(\bar\Gamma_{s'}^{(i)}))\cap V''_i=\mbox{int}(A(\bar\Gamma_{s'}^{(i)})\cap V_i'')=\emptyset$, and then $V''_i\subset Z_{s'}^{2}(i)$. That is, $Z_{s'}^2(i)$ is locally dense in  $V_i'$, which derives a contradiction to the definition of $k$.
This completes the proof of Lemma \ref{sw}.
\end{proof}

Next,  we consider a series of open sets $\lbrace S(i)\cap (-j,j)\rbrace_{j\geq 1}$ for $i=1,\ldots,n$. Clearly, $S(i)\cap (-j,j)\subset S(i)\cap (-(j+1),(j+1))$ and $\lim_{j\rightarrow+\infty}S(i)\cap (-j,j)=S(i)$. Then, by using Lemmas \ref{useful} and \ref{sw}, there is an integer $d\geq 1$ such that for any unit $x\in\mathbb{R}^{m}$,
 \begin{eqnarray}\label{ld>0}
\ell\left(\left\lbrace y\in \prod\nolimits_{i=1}^{n}(S(i)\cap(-d,d)): |\phi^{\tau}(y)x|>0\right\rbrace\right)>0.
\end{eqnarray}
Since $S(i)$ is open, for each integer $i\in[1,n]$, there exists some disjoint open intervals $\lbrace S_i^{j}\rbrace_{j\in\Theta_i}$, where
$\Theta_i=\lbrace 1,\ldots,k_i\rbrace$ ($k_i$ can be taken infinite),  such that $S(i)\cap(-d,d)=\bigcup_{j\in\Theta_i}S_i^{j}$. Write $S_{i}^{j}= (c_i^{j},d_i^{j})$ and denote
\begin{eqnarray}\label{Sijq}
S_{i}^{j}(q)\triangleq \left(c_i^j+\frac{d_i^j-c_i^j}{q+2},d_i^j-\frac{d_i^j-c_i^j}{q+2}\right),\quad   j\in\Theta_i,~q\in \mathbb{N}^+.
\end{eqnarray}
Given \dref{ld>0}, the following lemma is natural.
\begin{lemma}\label{qq}
If \dref{ld>0} holds, then  there exist some integers $p_1,\ldots,p_n$ and $q\geq 1$ such that for any unit $x\in\mathbb{R}^{m}$,
 \begin{eqnarray}\label{ely}
\ell\left(\left\lbrace y\in \mathcal{S}: |\phi^{\tau}(y)x|>0\right\rbrace\right)>0\quad \mbox{and}\quad \overline{\mathcal{S}}\subset \prod_{i=1}^{n} E_i,
\end{eqnarray}
where $\mathcal{S}\triangleq \prod\nolimits_{i=1}^{n}\bigcup\nolimits_{j=1}^{p_i}S_{i}^{j}(q)$.
\end{lemma}

\begin{proof}
Let $i\in[1,n]$. It is obvious that $\bigcup_{j\in\Theta_i}S_i^{j}(q)\subset \bigcup_{j\in\Theta_i}S_i^{j}(q+1)$, $q\in \mathbb{N}^+$.
If  $|\Theta_i|<+\infty$, then
\begin{equation*}
\lim_{q\rightarrow+\infty}\bigcup\nolimits_{j\in\Theta_i}S_i^{j}(q)=S(i)\cap(-d,d).
\end{equation*}
As for  the case where $\Theta_i=\mathbb{N}^{+}$, it infers
\begin{eqnarray*}
\lim_{k\rightarrow+\infty} \bigcup\nolimits_{j=1}^{k}S_i^{j}(k)=S(i)\cap(-d,d).
\end{eqnarray*}
So, in view of the above two cases, by \dref{ld>0} and  Lemma \ref{useful}, there are some integers $p_1,\ldots,p_n$ and $q\geq 1$ such that \dref{ely} holds.
\end{proof}

\subsubsection{Selection of $U_x$}
With the foregoing preliminaries in place, we can set out to construct $U_x$.
First, for every $x\in\mathbb{R}^{m}$ with $\|x\|=1$, define
\begin{equation*}
U_x(\delta)\triangleq \lbrace y: |\phi^{\tau}(y)x|> \delta\rbrace\cap \mathcal{S},\quad \delta>0.
\end{equation*}
The remaining task is to take a proper $\delta>0$ such that $U_x=U_x(\delta)$ meet our requirement. To this end,
let  $\lbrace d_{k}\rbrace_{k=1}^{2n}$ be a sequence  of numbers and for $k\in [n+1,2n]$,
define
\begin{eqnarray}\label{vard}
\varsigma_{k}\triangleq d_{k}-x^{\tau}\phi(d_{k-1},\ldots,d_{k-n}),\quad x\in\mathbb{R}^m.
\end{eqnarray}
Denote  $y=(d_{n},\ldots,d_{1})^{\tau}$ and
$\varsigma=(\varsigma_{2n},\ldots,\varsigma_{n+1})^{\tau}$. Evidently, \dref{vard} implies that there is a function $g: \mathbb{R}^{2n+m}\rightarrow \mathbb{R}^n$ such that
\begin{eqnarray}\label{d=g}
(d_{2n},\ldots,d_{n+1})^{\tau}=g(\varsigma, y,x).
\end{eqnarray}
We choose $\delta$  according to  the lemma below.

\begin{lemma}\label{inf>0}
Under  Assumption A2, the following two statements hold:\\
(i) given $y\in\mathbb{R}^{n}$, $x\in\mathbb{R}^{m}$ and
  a box $O=\prod_{i=1}^{n}I_{i}$ with  $\lbrace I_i\rbrace_{i=1}^{n}$ being some intervals, then
\begin{eqnarray}\label{lO}
\ell(\lbrace \varsigma: g(\varsigma,y,x)\in O\rbrace)=\ell(O);
\end{eqnarray}
(ii) for any constants $M,K>0$, there is  a $\delta^*>0$ such that
\begin{equation}\label{3i}
\inf_{\|z\|=1, \|y\|\leq M, \|x\|\leq K}\ell\left(\lbrace \varsigma: |\phi^{\tau}(g(\varsigma,y,x))z|> \delta^{*}, g(\varsigma,y,x)\in \mathcal{S}\rbrace\right)>0.
\end{equation}
\end{lemma}
\begin{proof}
(i) Note that in view of \dref{vard},
$
d_k=\varsigma_{k}+o_{k-1},  k=n+1,\ldots, 2n,
$
where $o_{k-1}\in \mathbb{R}$ is a point determined by   $\varsigma_{k-1}$,  $y$ and $x$ (for $k=n+1$, $\varsigma_n$ does not exist and $o_n$ depends  only on $y$ and $x$). So,
$\{\varsigma: \varsigma+o_{k-1}\in I_k\}=I_k-o_{k-1}$
is an interval with length $|I_k|$. By the definition of the	Lebesgue measure in $\mathbb{R}^n$, it is straightforward that
$$\ell(\lbrace \varsigma: g(\varsigma,y,x)\in O\rbrace)=\prod\nolimits_{k=1}^n |I_k|=\ell(O).$$
\\
(ii) Arguing by contradiction, we assume that \dref{3i} is false.
Then,  for each integer $k\geq 1$, there exists some point $(z(k),  y(k), x(k))  
$  falling in a compact set $\overline{B(0,1)}\times \overline{B(0,M)} \times \overline{B(0,K)}  \subset \mathbb{R}^{m}\times \mathbb{R}^{n}\times \mathbb{R}^{m}$ with $\|z(k)\|=1$ such that
\begin{equation}\label{l1k}
\ell(\lbrace \varsigma: |\phi^{\tau}(g(\varsigma,y(k),x(k)))z(k)|>\frac{1}{k}, g(\varsigma,y(k),x(k))\in \mathcal{S}\rbrace)<\frac{1}{k}.
\end{equation}
This sequence of points thus   has a subsequence $\lbrace z(k_r),  y(k_r), x(k_r)\rbrace_{r\geq 1}$  and an  accumulation point $(z^{*}, y^{*}, x^{*})$    such that
\begin{equation}\label{slwy}
\lim_{r\rightarrow +\infty}z(k_r)=z^{*},\quad \lim_{r\rightarrow +\infty}y(k_r)=y^{*},\quad \lim_{r\rightarrow +\infty}x(k_r)=x^{*}.   
\end{equation}
So, $\|z^{*}\|=1$, $\|y^{*}\|\leq M$, $\|x^{*}\|\leq K$. If
\begin{eqnarray}
\ell\left(\lbrace \varsigma: |\phi^{\tau}(g(\varsigma,y^{*},x^{*}))z^{*}|>0, g(\varsigma,y^{*},x^{*})\in \mathcal{S}\rbrace\right)=0,\nonumber
\end{eqnarray}
then  $\phi^{\tau}(y)z^{*}\equiv 0$ for all $y\in \mathcal{S}$ due to \dref{vard}, \dref{d=g} and the continuity of $\phi$. This clearly contradicts to Lemma \ref{qq}.    Therefore,  by \dref{Sijq},
\begin{eqnarray}
&&\lim_{k\rightarrow+\infty}\ell(\lbrace \varsigma: |\phi^{\tau}(g(\varsigma,y^{*},x^{*}))z^{*}|>\frac{1}{k}, g(\varsigma,y^{*},x^{*})\in  \mathcal{S}_k  \rbrace)\nonumber\\
&=&\ell\left(\lbrace \varsigma: |\phi^{\tau}(g(\varsigma,y^{*},x^{*}))z^{*}|>0, g(\varsigma,y^{*},x^{*})\in \mathcal{S}\rbrace\right)> 0,\nonumber
\end{eqnarray}
where $\mathcal{S}_k\triangleq  \prod_{i=1}^{n}\bigcup_{j=1}^{p_i} \left(c_i^j+\frac{d_i^j-c_i^j}{q+2}+\frac{1}{k},d_i^j-\frac{d_i^j-c_i^j}{q+2}-\frac{1}{k}\right)$. 
This implies that there exists an integer $h\geq 1$ such that
\begin{eqnarray}\label{ellh}
\ell(\lbrace \varsigma: |\phi^{\tau}(g(\varsigma,y^{*},x^{*}))z^{*}|>\frac{1}{h}, g(\varsigma,y^{*},x^{*})\in  \mathcal{S}_h \rbrace)>0.
\end{eqnarray}

Note that all  points  $\lbrace y(k_r), x(k_r)\rbrace_{r\geq 1}$  are restricted to $\overline{B(0,M)} \times \overline{B(0,K)} $,
\dref{vard} and \dref{d=g} then indicate that
there is a compact set $O'$ such that
$$\lbrace \varsigma: g(\varsigma,y(k_r),x(k_r))\in \mathcal{S}\rbrace\subset O'.$$
Further, $g$ and $\phi$ are continuous
due to \dref{vard}, \dref{d=g} and Assumption A2(i), hence \dref{slwy} shows
\begin{eqnarray*}\label{sl1}
\left\{
\begin{array}{l}
\lim\limits_{r\rightarrow\infty}\sup\limits_{\varsigma\in O'}\|g(\varsigma, y^*,x^*)-g(\varsigma, y(k_r),x(k_r))\|=0\\
\lim\limits_{r\rightarrow\infty}\sup\limits_{\varsigma\in O'}\|\phi^{\tau}(g(\varsigma,y^*,x))z^*-\phi^{\tau}(g(\varsigma,y(k_r),x(k_r)))z(k_r)\|=0
\end{array}.
\right.
\end{eqnarray*}
As a consequence, for all sufficiently large $r$,
\begin{eqnarray}
&&\ell(\lbrace \varsigma: |\phi^{\tau}(g(\varsigma,y^{*},x^{*}))z^{*}|>\frac{1}{h}, g(\varsigma,y^{*},x^{*})\in  \mathcal{S}_h \rbrace)\nonumber\\
&<&\ell(\lbrace \varsigma: |\phi^{\tau}(g(\varsigma,y(k_r),x(k_r)))z(k_r)|>\frac{1}{k_r}, g(\varsigma,y(k_r),x(k_r))\in \mathcal{S}\rbrace)<\frac{1}{k_r},\nonumber
\end{eqnarray}
which contradicts to \dref{ellh} by letting $r\rightarrow+\infty$.   Lemma  \ref{inf>0} thus follows.
\end{proof}

\begin{remark}
In  Lemma \ref{inf>0},  Assumption A2 can be weaken to Assumption A2' when $n=1$. Statement (i) is trivial.   For (ii),  note that \dref{ely} still holds by Assumption A2'. But, \dref{vard}, \dref{l1k} and \dref{ellh} yield that for all   sufficiently large $r$,
\begin{eqnarray*}
\frac{1}{k_r}
&> &\ell(\lbrace \varsigma: |\phi^{\tau}(g(\varsigma,y(k_r),x(k_r)))z(k_r)|>\frac{1}{k_r},g(\varsigma,y(k_r),x(k_r))\in \mathcal{S}\rbrace)\nonumber\\
&=&\ell(\lbrace y: |\phi^{\tau}(y)z(k_r)|>\frac{1}{k_r}, y\in \mathcal{S}\rbrace)\\
&\geq& \ell(\lbrace y\in \mathcal{S}: |\phi^{\tau}(y)z^{*}|>\frac{1}{k_r}+\frac{1}{h}\rbrace),\nonumber
\end{eqnarray*}
where  $\lbrace z(k_r),  y(k_r), x(k_r)\rbrace_{r\geq 1}$ is defined in the  proof of Lemma \ref{inf>0}.    Letting $r\rightarrow+\infty$ in the above inequality infers
\begin{eqnarray}
0
&\geq &\lim_{r\rightarrow+\infty}\ell(\lbrace y\in \mathcal{S}: |\phi^{\tau}(y)z^{*}|>\frac{1}{k_r}+\frac{1}{h}\rbrace)\nonumber\\
&=&\ell(\lbrace y\in \mathcal{S}: |\phi^{\tau}(y)z^{*}|>\frac{1}{h}\rbrace),
\end{eqnarray}
which contradicts to \dref{ellh}.
\end{remark}

At the end of this section, fix two numbers  $M$ and $K$.  According to Lemma \ref{inf>0}(ii),  we select a  $\delta^{*}$ such that \dref{3i} holds. Now, for any unit vector $x\in\mathbb{R}^{m}$, define $U_{x}\triangleq U_x(\delta^{*})$.

\subsection{\textbf{The Properties of $U_x$}}\label{32}

To analyze the properties of $U_x$, we first prove a lemma below.

\begin{lemma}\label{imp}
Fix an integer $i\in[1,n]$. Let  $x_i=(x_{i1},\ldots,x_{im_i})^{\tau}\in \mathbb{R}^{m_i}$ be a non-zero vector and $d>c$ be two  numbers satisfying $[c,d]\subset E_i$. Also, let $\lbrace r_l\rbrace_{l=1}^{2^{m_i-1}}$ be a sequence of numbers that   $d\geq r_1>r_2>\ldots>r_{2^{m_i-1}}\geq c$ and
\begin{equation}\label{2n}
\sum_{j=1}^{m_i}f_{ij}'(r_l)x_{ij}=0,\quad 1\leq l\leq 2^{m_i-1},
\end{equation}
where $f_{ij}'=Df_{ij}, j=1,\ldots,m_i$.
Then, the following two statements hold:\\
(i) there exists an array $s\in \mathcal{H}_i$ such that 
\begin{equation*}
A(\bar\Gamma_{s}^{(i)})\cap [c,d]\neq\emptyset;
\end{equation*}
(ii) if for every $s\in \mathcal{H}_i$,
$A(\bar\Gamma_{s}^{(i)})\cap [c,d] $  is either $\emptyset$ or $[c,d]$,    
then
\begin{equation}\label{fcd}
\sum_{j=1}^{m_i}f_{ij}'(y)x_{ij}=0,\quad \forall y\in [c,d].
\end{equation}
\end{lemma}
\begin{proof}
(i) Suppose $\bigcup_{s\in \mathcal{H}_i}A(\bar\Gamma_{s}^{(i)})\cap [c,d]=\emptyset$. Then,  for each integer $k\in[1,m_i]$, there exist $2^{k-1}$ numbers $\lbrace \varepsilon_{k,l}\rbrace_{l=1}^{2^{k-1}}$ satisfying $d\geq \varepsilon_{k,1}>\ldots>\varepsilon_{k,2^{k-1}}\geq c$ and
\begin{equation}\label{kx}
\sum_{j=1}^{k}\bar\Gamma_{(j,k+1,\ldots,m_i)}^{(i)}(\varepsilon_{k,l})x_{ij}=0,\quad 1\leq l\leq 2^{k-1},
\end{equation}
where $\bar\Gamma_{(j,m_i+1,\ldots, m_i)}^{(i)}\triangleq f_{ij}'.$

As a matter of fact,
when $k=m_i$, \dref{2n} leads to \dref{kx} immediately. We now prove \dref{kx}  by induction.   Assume \dref{kx} holds for $k=m'$, where  $m'$ is an integer in $[2,m_i]$. Hence we can find $2^{m'-1}$ numbers of $\lbrace \varepsilon_{m',l}\rbrace_{l=1}^{2^{m'-1}}$ such that $d\geq \varepsilon_{m',1}>\ldots>\varepsilon_{m',2^{m'-1}}\geq c$  and
\begin{equation}
\sum_{j=1}^{m'}\bar\Gamma_{(j,m'+1,\ldots,m_i)}^{(i)}(\varepsilon_{m',l})x_{ij}=0,\quad 1\leq l\leq 2^{m'-1}.\nonumber
\end{equation}
By $\bigcup_{s\in \mathcal{H}_i}A(\bar\Gamma_{s}^{(i)})\cap [c,d]=\emptyset$, every $\Gamma_{(j,m',\ldots,m_i)}^{(i)}$ is well-defined in $[c,d]$, then for $1\leq l\leq 2^{m'-1}$,
\begin{equation}
\sum_{j=1}^{m'-1}\Gamma_{(j,m',\ldots,m_i)}^{(i)}(\varepsilon_{m',l})x_{ij}=\sum_{j=1}^{m'-1}\frac{\bar\Gamma_{(j,m'+1,\ldots,m_i)}^{(i)}}{\bar\Gamma_{(m',\ldots,m_i)}^{(i)}}(\varepsilon_{m',l})x_{ij}=-x_{im'}.\nonumber
\end{equation}
Taking account of the \textit{Rolle's theorem}, there are some $\lbrace \varepsilon_{m'-1,l}\rbrace_{l=1}^{2^{m'-2}}$ with
$\varepsilon_{m'-1,l}\in(\varepsilon_{m',2l-1},\varepsilon_{m',2l})$ such that
\begin{equation*}
\sum_{j=1}^{m'-1}\bar\Gamma_{(j,m',\ldots,m_i)}^{(i)}(\varepsilon_{m'-1,l})x_{ij}=0,\quad 1\leq l\leq 2^{m'-2}.
\end{equation*}
Therefore, \dref{kx} holds for $k=m'-1$ and this completes the induction.

Now, by letting $k=1$ in \dref{kx}, there is a number $\varepsilon_{1,1}\in[c,d]$ such that $\bar\Gamma_{(1,\ldots,m_i)}^{(i)}(\varepsilon_{1,1})x_{i1}=0$. Since $\bigcup_{s\in \mathcal{H}_i}A(\bar\Gamma_{s}^{(i)})\cap [c,d]=\emptyset$,  $\bar\Gamma_{(1,\ldots,m_i)}^{(i)}(\varepsilon_{1,1})\neq 0$, and hence  $x_{i1}=0$. By the symmetry of $x_{i1},\ldots,x_{im_i}$ in \dref{kx}, we conclude that  $x_{ij}=0$ for all $ j\in [1,m_i]$. But this  is impossible due to $\|x_i\|\not=0$ and thus $\bigcup_{s\in \mathcal{H}_i}A(\bar\Gamma_{s}^{(i)})\cap [c,d]\neq\emptyset$.\\
(ii) Let $I$ be an open interval containing $[c,d]$.  It suffices to prove the claim that for every  function sequence  $f_{i1},\ldots,f_{im_i}\in C^{m_i}(I)$ satisfying \dref{2n},
if $A(\bar\Gamma_{s}^{(i)})\cap [c,d] $  is either $\emptyset$ or $[c,d]$,  $\forall  s\in \mathcal{H}_i$, then  \dref{fcd}  holds. We show it by induction.
When $m_i=1$,  \dref{2n}   reduces to  $f_{i1}'(r_1)x_{i1}=0$. Since $x_{i1}\not=0$,  $A(f_{i1}')\cap[c,d]\not=\emptyset$, which means $A(f_{i1}')\supset[c,d]$ by assumption. So,  $f_{i1}'(y)x_{i1}\equiv0$ for all $y\in[c,d]$.
Suppose  the   claim  mentioned above holds for  all    $m_i\in[1,h-1]$, $h\geq 2$.

We now consider the claim for  $m_i=h$. In this case,  the non-zero vector $x_i=(x_{i1},\ldots,x_{ih})^{\tau}$. First, assume that there is an integer $j'\in[1,h]$ such that $|x_{ij'}|<\|x_i\|$ and
\begin{equation}\label{afj}
A(f_{ij'}')\cap [c,d]=\emptyset.
\end{equation}
Without loss of generality, let $j'=h$. 
Define the following $h-1$ functions:
\begin{eqnarray*}
F_{j}\triangleq{\Gamma_{(j,h)}^{(i)}},\quad 1\leq j \leq h-1.
\end{eqnarray*}
Owing to \dref{afj}, $F_j\in C^{h-1}(I)$, $1\leq j \leq h-1 $ with $  h\geq 2$ are  well-defined. 
 Moreover, \dref{2n} yields
\begin{equation}\label{fir}
\sum_{j=1}^{h-1}F_{j}(r_l)x_{ij}=-x_{ih},\quad 1\leq l\leq 2^{h-1}.
\end{equation}
Therefore, by applying the \textit{Rolle's theorem}, there exist $2^{h-2}$ numbers $\varepsilon_{l}\in[r_{2l},r_{2l-1}]$, $l\in[1,2^{h-2}]$ such that
\begin{equation}\label{F'e}
\sum_{j=1}^{h-1}DF_{j}(\varepsilon_l)x_{ij}=0,\quad 1\leq l\leq 2^{h-2}.
\end{equation}
Here, $(x_{i1},\ldots,x_{i(h-1)})^{\tau}$ is nonempty by $|x_{ih}|<\|x_i\|$.

Since for every $(i_1,\ldots,i_{l})\in\bigcup_{k=1}^{h-1} \mathcal{H}_k^{(1,\ldots,h-1)}$, $(i_1,\ldots,i_{l},h)\in \bigcup_{k=1}^{h} \mathcal{H}_k^{(1,\ldots,h)}$, then by \dref{Lambdal},
\begin{equation}\label{mj1}
\Lambda_{l}(F_{i_1},\ldots,F_{i_{l}})=\Lambda_{l+1}(f_{ii_1},\ldots,f_{ii_{l}},f_{ih})=\Gamma_{(i_1,\ldots,i_{l},h)}^{(i)}.
\end{equation}
Because  $A(\bar\Gamma_{s}^{(i)})\cap [c,d] $  is either $\emptyset$ or $[c,d], \forall s\in \bigcup_{k=1}^{h} \mathcal{H}_k^{(1,\ldots,h)}$, \dref{mj1} yields
\begin{equation*}
A\left(D\Lambda_{l}(F_{i_1},\ldots,F_{i_{l}})\right)\cap [c,d]=\emptyset~\mbox{or} ~[c,d].
\end{equation*}
Consequently, by the induction hypothesis with $m_i=h-1$ and   $f_{ii_j}=F_j, j\in [1,h-1]$ satisfying \dref{F'e}, we  conclude
\begin{equation}\label{fiy}
\sum_{j=1}^{h-1}DF_{j}(y)x_{ij}=0,\quad \forall y\in [c,d].
\end{equation}
In view of \dref{fir} and \dref{fiy}, we  deduce that
$\sum_{j=1}^{h-1}F_{j}(y)x_{ij}=-x_h$ for any $ y\in[c,d]$,
and hence
\begin{equation}\label{mfy}
\sum_{j=1}^{h}f_{ij}'(y)x_{ij}\equiv0 \quad \mbox{for all}\quad  y\in [c,d].
\end{equation}

Now, it remains to consider the case that for each integer
 $j\in[1,h]$, either  $|x_{ij}|=\|x_{i}\|$ or $[c,d]\subset A(f_{ij}')$.
 If $|x_{ij}|<\|x_{i}\|$ for all $j\in[1,h]$, then $f_{ij}'\equiv 0$ in $[c,d]$ for all $j\in[1,h]$, which leads to \dref{mfy}. So, assume there is an integer $j'\in[1,h]$  that $|x_{ij'}|=\|x_i\|$. Without loss of generality, let $j'=h$, then $x_{ij}=0$ for all $j\in[1,h-1]$.   Substituting this into  \dref{2n}, one has
\begin{equation*}
f_{ih}'(r_l)=0,\quad 1\leq l\leq 2^{h-1}.
\end{equation*}
The induction hypothesis thus yields $[c,d]\subset A(f_{ih}')$, and hence
\begin{equation*}
\sum_{j=1}^{h}f_{ij}'(y)x_{ij}=f_{ih}'(y)x_{ih}=0,\quad \forall y\in [c,d].
\end{equation*}
Therefore, the claim is
 true for $m_i=h$ and   we complete the induction.
\end{proof}
We now return to analyze $Z_s^{1}(i)\cap ( \bigcup_{j=1}^{p_i}S_{i}^{j}(q))$. Observe  that for each array $s\in\mathcal{H}_i$, if $Z_{s}^{1}(i)\cap(\bigcup_{j=1}^{p_i}S_{i}^{j}(q))\neq\emptyset$, it is  a countable union  of disjoint open intervals. Denote the set of these intervals by $\mathcal{G}_s(i)\triangleq \lbrace I_{s}^{j}(i)\rbrace_{j\geq 1}$, where
\begin{eqnarray}
I_{s}^{j}(i)=(a_{s}^{j}(i),b_{s}^{j}(i)),\quad j=1,2,\ldots.
\end{eqnarray}
Let
$\mathcal{G}(i)\triangleq {\bigcup_{ s\in  \mathcal{H}_i}\mathcal{G}_s(i)}$ for each $i\in[1,n].$
Furthermore, define
\begin{eqnarray*}
H(i)\triangleq \left(\bigcup_{s\in\mathcal{H}_i}Z_{s}^{3}(i)\right)\cap \left(\bigcup_{j=1}^{p_i}S_{i}^{j}(q)\right),\quad i\in[1,n].
\end{eqnarray*}

\begin{lemma}
For each $i\in[1,n]$,
\begin{equation}\label{card}
|\mathcal{G}(i)|<+\infty \quad \mbox{and}\quad |H(i)|<+\infty.
\end{equation}
\end{lemma}
\begin{proof}
Suppose $|\mathcal{G}(i)|=+\infty$ for some $i\in[1,n]$, then there is an array $s^{*}\in  \mathcal{H}_i$ such that $|\mathcal{G}_{s^{*}}(i)|=+\infty$. Let $y\in\bigcup_{j=1}^{p_i}\overline{S_{i}^{j}(q)}$ be an accumulation point of $\lbrace b_{s^{*}}^{j}(i)\rbrace_{j\geq 1}$. By the continuity of $ \bar\Gamma_{s^{*}}^{(i)}$ in set $ \bigcup_{j=1}^{p_i}\overline{S_{i}^{j}(q)}\subset E_i$, $y\in\bigcup_{j=1}^{p_i}\overline{S_{i}^{j}(q)}\cap A(\bar\Gamma_{s^{*}}^{(i)})$. Moreover, it is evident that  $y\not\in Z_{s^{*}}^{1}(i)\cup Z_{s^{*}}^{3}(i)$, so
$$y\in Z_{s^{*}}^{2}(i)\cap\left(\bigcup_{j=1}^{p_i}\overline{S_{i}^{j}(q)}\right)\subset E_i\cap\left(\bigcup_{s\in\mathcal{H}_i}Z_{s}^{2}(i)\right).$$
 However,
\begin{eqnarray}\label{dist}
E_i\cap\left(\bigcup_{s\in\mathcal{H}_i}Z_{s}^{2}(i)\right)=E_{i}\backslash S(i)\subset E_{i}\backslash \bigcup_{j=1}^{p_i}S_i^{j},
\end{eqnarray}
and
\begin{eqnarray}\label{dd}
\left(E_{i}\backslash \bigcup_{j=1}^{p_i}S_i^{j}\right)\cap\left( \bigcup_{j=1}^{p_i}\overline{S_{i}^{j}(q)}\right)=\emptyset.
\end{eqnarray}
The contradiction is derived immediately by comparing \dref{dist}, \dref{dd} and the fact $y\in E_i\cap\left(\bigcup_{s\in\mathcal{H}_i}Z_{s}^{2}(i)\right)$. Thus, $|\mathcal{G}(i)|<+\infty$.

As to $|H(i)|<+\infty$, the proof  is quite similar to that given  for $|\mathcal{G}(i)|<+\infty$ and  is omitted.
\end{proof}

The following lemma is based on the above two lemmas.
\begin{lemma}\label{sdf}
Given $i\in [1,n]$, let  $x_i\in\mathbb{R}^{m_i}$ be  a non-zero vector.  Denote $(\phi^{(i)})'=(f_{i1}',\ldots,f_{im_i}')^{\tau}$ and
\begin{eqnarray*}
\left\{
\begin{array}{l}
K_{i}=\mbox{int}(A(x_{i}^{\tau}(\phi^{(i)})'))\cap\left(\bigcup_{j=1}^{p_i}\overline{S_i^{j}(q)}\right)\\
L_{i}=(A(x_{i}^{\tau}(\phi^{(i)})')\cap\left(\bigcup_{j=1}^{p_i}\overline{S_i^{j}(q)})\right)\setminus K_{i}
\end{array},
\right.
\end{eqnarray*}
then $|L_{i}|\leq 2^{m_i}p_i(3|\mathcal{G}(i)|+|H(i)|+2)$.
\end{lemma}
\begin{proof}
Let $\mathcal{Q}_{ij}\triangleq (A(x_{i}^{\tau}(\phi^{(i)})')\cap\overline{S_i^{j}(q)})\setminus K_{i}$, $j=1,\ldots,p_i$. We first show that the cardinality of each $\mathcal{Q}_{ij}$ is finite. Otherwise, for some $j\in [1,p_i]$,  there is  a monotone sequence   $\lbrace r_{l}\rbrace_{l\geq 1}$ in $\mathcal{Q}_{ij}$ such that $r_l\neq r_{l'}$ for each $l\neq l'$ and $\lim_{l\rightarrow+\infty}r_l=y^*$ for some $y^{*}\in\overline{S_i^{j}(q)}$.  Without loss of generality, let
$r_{l}<r_{l'}$ if $l>l'$.
Divide this sequence into infinite groups:
\begin{equation*}
\left\lbrace r_{2^{m_i}k+1},r_{2^{m_i}k+2},\ldots,r_{2^{m_i}(k+1)}\right\rbrace,\quad k=0,1,\ldots
\end{equation*}
and  for each $k\geq 0$,  define
\begin{equation}\label{vd}
D_{k}\triangleq{[r_{2^{m_i}(k+1)},r_{2^{m_i}k+1}]}.
\end{equation}
So, given $k\geq 0$,  $D_k\subset \overline{S_i^{j}(q)}\subset E_i$   
and $r_{2^{m_i}k+1}>\ldots>r_{2^{m_i}(k+1)}$ satisfy
\begin{equation}\label{rk}
x_i(\phi^{(i)})^{'}(r_l)=0,\quad 2^{m_i}k+1\leq l\leq 2^{m_i}(k+1).
\end{equation}
Note that the definition of $\mathcal{Q}_{ij}$ yields
$x_i(\phi^{(i)})^{'}\not\equiv 0$ on $D_k$,
 applying Lemma \ref{imp} with $[c,d]=D_k$ indicates that  there is an array $s_k\in\mathcal{H}_i$ fulfilling $A(\bar\Gamma_{s_k}^{(i)})\cap D_k\neq \emptyset$ and
  $  D_k \not\subset  A(\bar\Gamma_{s_k}^{(i)})$. Hence, at least one of  following three cases  occurs:\\
\emph{Case 1:} $Z_{s_k}^{3}(i)\cap D_k\neq \emptyset$.\\
\emph{Case 2:} There is an interval $I_{s_k}^{j}(i)\in \mathcal{G}_{s_k}(i)$ such that   $I_{s_k}^{j}(i)\subset D_k$.\\
\emph{Case 3:} There is an interval $I_{s_k}^{j}(i)\in \mathcal{G}_{s_k}(i)$ satisfying $I_{s_k}^{j}(i)\cap D_k\neq \emptyset$ and $ D_k   \not\subset    I_{s_k}^{j}(i)$.

Let $O_{l}, l=1,2,3$ present the times of
 Case $l$ that occurs for some  $k\geq 0$.  Since
$D_{k}\cap D_{k'}=\emptyset$ for $k\neq k'$,
\begin{eqnarray}\label{oo1}
O_{1}\leq |H(i)| \quad \mbox{and}\quad
O_{2}\leq |\mathcal{G}(i)|.
\end{eqnarray}
Furthermore, for each interval  $I_{s}^{j}(i)\in G_{s}(i)$, there exist at most two distinct $D_k$ such that $I_{s}^{j}(i)\cap D_k\neq \emptyset$ and $ D_k   \not\subset    I_{s}^{j}(i)$. Hence,
\begin{equation}\label{o1}
O_{3}\leq 2|\mathcal{G}(i)|.
\end{equation}
Combining \dref{oo1} and \dref{o1} yields
\begin{equation}\label{o123}
O_1+O_{2}+O_3\leq 3|\mathcal{G}(i)|+|H(i)|<+\infty.
\end{equation}
However,  $O_1+O_{2}+O_3=+\infty$ because $D_k$ is infinite. So, the cardinality of each $\mathcal{Q}_{ij}$ is  finite.

Now, let  $j\in[1,p_i]$ be an index  that
\begin{eqnarray}
|\mathcal{Q}_{ij}|>2^{m_i}(3|\mathcal{G}(i)|+|H(i)|+2).\nonumber
\end{eqnarray}
Then, write the points of  $\mathcal{Q}_{ij}$ from left to right as $v_{0}, v_1, \ldots,v_{|\mathcal{Q}_{ij}|-1}$.
Define
\begin{equation}\label{hca}
h\triangleq{\left\lfloor \frac{|\mathcal{Q}_{ij}|-2}{2^{m_i}}\right\rfloor}> 3|\mathcal{G}(i)|+|H(i)|
\end{equation}
and
\begin{equation*}
D_{k}^{'}\triangleq{[v_{2^{m_i}(k+1)},v_{2^{m_i}k+1}]},\quad k=0,1,\ldots,h-1.
\end{equation*}
by an analogous proof as  \dref{rk}--\dref{o123}, we  arrive at
\begin{equation*}
h\leq 3|\mathcal{G}(i)|+|H(i)|,
\end{equation*}
which arises a contradiction to \dref{hca}. Therefore, $|\mathcal{Q}_{ij}|\leq 2^{m_i}(3|\mathcal{G}(i)|+|H(i)|+2)$ and hence
$|L_{i}|\leq 2^{m_i}p_i(3 |\mathcal{G}(i)|+|H(i)|+2)$.
\end{proof}

For any $i\in[1,n]$,  $x_i\in\mathbb{R}^{m_i}$ and $\delta\in\mathbb{R}$, it is clear that  set  $\lbrace y: x^{\tau}_i\phi^{(i)}(y)> \delta\rbrace\cap (\bigcup_{j=1}^{p_i}S_i^{j}(q))$ is   open. If this set is not empty, then it is          a  countable  union of disjoint open intervals.  Denote the set of these intervals by  $\mathcal{U}_i(\delta)$.

\begin{lemma}\label{ldfb}
Let  $i\in[1,n]$. Then,  for any non-zero $x_i\in\mathbb{R}^{m_i}$     and $\delta\in\mathbb{R}$,
\begin{equation}\label{UiL}
|\mathcal{U}_i(\delta)|\leq p_i(|L_i|+2).
\end{equation}
\end{lemma}
\begin{proof}
 Denote
\begin{equation*}
\mathcal{K}_{ij}\triangleq{\lbrace I\in \mathcal{U}_i(\delta): I\cap S_i^{j}(q)\neq\emptyset \rbrace},\quad \forall j\in[1,p_i],
\end{equation*}
then
\begin{equation}\label{pca}
|\mathcal{U}_i(\delta)|\leq \sum_{j=1}^{p_i} |\mathcal{K}_{ij}|.
\end{equation}

Fix an index $j\in[1,p_i]$ and $I\in \mathcal{K}_{ij}$. By the continuity of $\phi^{(i)}$ in $S_i^{j}(q)$, each
 endpoint of $I$ either  belongs to the zero set $A(x^{\tau}_i\phi^{(i)}(y)-\delta)$ or is an endpoint of  $S_i^{j}(q)$. If $\partial(I)\cap\partial(S_i^{j}(q))=\emptyset$, then $\partial(I)\in A(x^{\tau}_i\phi^{(i)}(y)-\delta)$.  By the \textit{Rolle's theorem}, it follows that $\lbrace y: x^{\tau}_i(\phi^{(i)})'(y)=0\rbrace\cap I\not=\emptyset$, which together with $I\subset\lbrace y: x^{\tau}_i\phi^{(i)}(y)> \delta\rbrace$ leads to $L_i\cap I\not=\emptyset$. Note that there are at most two intervals $I
 \in \mathcal{K}_{ij}$ satisfying $\partial(I)\cap \partial(S_i^{j}(q))\not=\emptyset$ and any two intervals in $\mathcal{K}_{ij}$ are disjoint, so
\begin{eqnarray}\label{kl2}
|\mathcal{K}_{ij}|\leq |L_i|+2.
\end{eqnarray}
Finally, \dref{UiL} is an immediate result of  \dref{pca} and \dref{kl2}.
\end{proof}

Given a closed box $O=\prod_{i=1}^{n}I_{i}\in\mathbb{R}^{n}$ and a positive integer $r$, equally divide each $I_i$ into $r$ closed intervals  $\lbrace I_{i,j}\rbrace_{j=1}^{r}$ that $\mbox{int}(I_{i,j})\cap \mbox{int}(I_{i,j'})=\emptyset  $ if $j\neq j'$. So, there are
$r^{n}$  small closed   boxes $\prod_{i=1}^{n}\lbrace I_{i,j}\rbrace _{j=1}^{r}$. Let $\mathcal{T}(O,r)$ be the set of the $r^n$ small boxes.   Clearly, for any distinct boxes $U, U'\in  \mathcal{T}(O,r)$,   $ \mbox{int}(U)\cap \mbox{int}(U') =\emptyset  $. Define
\begin{eqnarray}\label{kdef}
\mathcal{T}_\d(O,r)  \triangleq \left\{ U\in \mathcal{T}(O,r): \mathcal{B}(\delta)  \cap \overline {\mathcal{S} } \cap U \neq \emptyset \right   \},
\end{eqnarray}
 where $\mathcal{B}(\delta)\triangleq \partial(\lbrace y: \phi^{\tau}(y)x> \delta\rbrace)$ and $\mathcal{S}$ is defined in Lemma \ref{qq}. Let $\mathcal{K}_{\delta}(O,x,r)\triangleq |\mathcal{T}_\d(O,r)|$.
 The following lemma is critical to our result.
\begin{lemma}\label{xy}
There is a constant $C>0$ such that for any closed box $O=\prod_{i=1}^{n}I_{i}$, non-zero vector $x\in\mathbb{R}^{m}$, $\delta\in\mathbb{R}$ and integer $r\geq 1$,
\begin{eqnarray}\label{kox}
\mathcal{K}_{\delta}(O,x,r)\leq Cr^{n-1}.
\end{eqnarray}
\end{lemma}

\begin{proof}
We prove \dref{kox} by   induction. For $n=1$, let $O=I_1$ be a closed box.  By Lemma \ref{ldfb} with $n=1$, it is easy to check that
\begin{eqnarray}\label{bd2}
\left|\mathcal{B}(\delta)\cap\left( \bigcup_{j=1}^{p_1}S_1^{j}(q)\right)\right|\leq 2p_1(|L_1|+2).
\end{eqnarray}
Moreover, since
\begin{eqnarray*}
\mathcal{B}(\delta)\cap \left(\bigcup_{j=1}^{p_1}\overline{S_1^{j}(q)}\right)
\subset \mathcal{B}(\delta)\cap \left(\bigcup_{j=1}^{p_1}S_1^{j}(q)\right)\cup\partial\left(\bigcup_{j=1}^{p_1}S_1^{j}(q)\right),
\end{eqnarray*}
it follows that
\begin{eqnarray*}
\mathcal{K}_{\delta}(O,x,r)\leq 2\left|\mathcal{B}(\delta)\cap\left( \bigcup_{j=1}^{p_1}\overline{S_1^{j}(q)}\right)\right|
\leq 4p_1(|L_1|+2)+4p_1.
\end{eqnarray*}
Hence, \dref{kox} is true for $n=1$ by taking $C=4p_1(|L_1|+2)+4p_1$.

Now, suppose  \dref{kox} holds for $n=k$ with some $ k\geq 1$. Let us consider the case where $n=k+1$.
Take a closed box $O=\prod_{i=1}^{k+1}I_{i}\in\mathbb{R}^{k+1}$, and let $\mathcal{T}(O,r)$ be the set of the
$r^{k+1}$ disjoint refined  boxes.  These boxes correspond to two sets
$$
\mathcal{T}^1=\prod_{i=1}^{k}\lbrace I_{i,j}\rbrace_{j=1}^{r}\quad \mbox{and}\quad \mathcal{T}^2=\lbrace I_{k+1,j}\rbrace_{j=1}^{r}.
$$

Write vector $x=\mbox{col}\lbrace x_1,\ldots,x_{k+1}\rbrace\not=\textbf{0}$. First,  assume  there is an index $l\in [1,k+1]$ such that $x_l=\textbf{0}$. Without loss of generality, let $l=k+1$, then
\begin{eqnarray}\label{pak}
&&\mathcal{B}(\delta)\cap\prod_{i=1}^{k+1}\bigcup_{j=1}^{p_i}\overline{S_i^{j}(q)}\cap O \subset \nonumber\\
&&\bigg(\partial\bigg(\bigg\lbrace z\in\mathbb{R}^{k}: \sum_{i=1}^{k}x_{i}\phi^{(i)}(z_i)> \delta\bigg\rbrace\bigg)\cap \prod_{i=1}^{k}\bigcup_{j=1}^{p_i}\overline{S_i^{j}(q)}\cap\prod_{i=1}^{k}I_i\bigg)\times I_{k+1}.
\end{eqnarray}
where
\begin{eqnarray}\label{zdef}
z=(z_1,\ldots,z_{k})^{\tau}\in \mathbb{R}^k.
\end{eqnarray}
By applying the induction assumption  for $n=k$ and for  the refined boxes in $\mathcal{T}^1$, there is a constant $C>0$ such that
 $$\mathcal{K}_\d\left(\prod_{i=1}^{k}I_{i}, \mbox{col}\lbrace x_1,\ldots,x_{k}\rbrace, r \right)\leq Cr^{k-1},$$
 which, together with \dref{pak} and $\mathcal{T}(O,a)=\mathcal{T}^1\times \mathcal{T}^2$, yields $\mathcal{K}_\d(O,x, r)\leq Cr^{k}$. This is exactly \dref{kox} for $n=k+1$.

So, let $x_{i}\not=\textbf{0}$ for all $i\in[1,k+1]$. For any $B\in \mathcal{T}^1$, define set
\begin{eqnarray}
Z(B)\triangleq\left\lbrace z_{k+1}\in I_{k+1}: (B\times z_{k+1})\cap \mathcal{B}(\delta)\cap \prod_{i=1}^{k+1}\bigcup_{j=1}^{p_i}\overline{S_i^{j}(q)}\not=\emptyset \right \rbrace.\nonumber
\end{eqnarray}
Observe that $Z(B)$ is a closed set, then $\partial{Z(B)}\subset Z(B)$. Define
\begin{eqnarray*}
\left\{
\begin{array}{l}
\mathcal{Z}_1(B)\triangleq \{I_{k+1,j}\in \mathcal{T}^2: Z(B)\cap I_{k+1,j}\not=\emptyset\}\\
\mathcal{Z}_2(B)\triangleq \{I_{k+1,j}\in \mathcal{T}^2: \partial{Z(B)}\cap I_{k+1,j}\not=\emptyset\}
\end{array}.
\right.
\end{eqnarray*}
Since any interval in $\mathcal{Z}_{1}(B)\setminus \mathcal{Z}_2(B)$  must be contained in $Z(B)$,
\begin{eqnarray*}
|\mathcal{Z}_{1}(B)|-|\mathcal{Z}_2(B)|=|\mathcal{Z}_{1}(B)\setminus \mathcal{Z}_2(B)|\leq \frac{r}{|I_{k+1}|}\ell(Z(B)).
\end{eqnarray*}
At the same time,
\begin{eqnarray*}
\sum_{B\in \mathcal{T}^{1}}\ell(Z(B))
&=&\sum_{B\in \mathcal{T}^{1}}\int_{\mathbb{R}}I_{Z(B)}dz_{k+1}
\nonumber\\
&=& \int_{I_{k+1}}\sum_{B\in \mathcal{T}^{1}}I_{Z(B)}dz_{k+1},
\end{eqnarray*}
therefore
\begin{eqnarray}\label{gj1}
\mathcal{K}_\d(O,x,r)&=&\sum_{B\in \mathcal{T}^{1}}|\mathcal{Z}_{1}(B)|\nonumber\\
&\leq &\frac{r}{|I_{k+1}|}\int_{I_{k+1}}\sum_{B\in \mathcal{T}^{1}}I_{Z(B)}dz_{k+1}+\sum_{B\in \mathcal{T}^{1}}|\mathcal{Z}_{2}(B)|.
\end{eqnarray}
The last step is to  estimate  the  term in \dref{gj1}. Since the argument is involved, it is included in  Appendix \ref{AppA}.
In light of
 Lemmas \ref{ap2} and \ref{ap3}, when $n=k+1$, there are two constants $C_1, C_2>0$ depending only on $\phi$ such that
 \begin{eqnarray}
\mathcal{K}_{\delta}(O,x,r)\leq \left(C_1+C_2\right) r^{k}.\nonumber
\end{eqnarray}
The proof is thus completed.
\end{proof}

Now, recall the definition of $U_x$ in the end of Section \ref{cu},
\begin{eqnarray}
\partial(U_x)\subset((\partial(\lbrace y: \phi^{\tau}(y)x> \delta^{*}\rbrace)\cup \partial(\lbrace y: -\phi^{\tau}(y)x> \delta^{*}\rbrace))\cap \mathcal{S})\cup \partial(\mathcal{S}).\nonumber
\end{eqnarray}
Given a closed box $O$ and an integer $r\geq 1$, observe that
$$
|\lbrace U\in \mathcal{T}(O,r): \partial(\mathcal{S})\cap U\not=\emptyset\rbrace|\leq 4r^{n-1}\sum_{i=1}^{n}p_{i}.
$$
In addition,  by applying Lemma \ref{xy}
it follows that there is a constant $C_0>0$  depending only on  $\phi$ such that
\begin{eqnarray}\label{uor}
|\lbrace U\in \mathcal{T}(O,r): \partial(U_x)\cap U\not=\emptyset\rbrace|\leq C_0r^{n-1}.
\end{eqnarray}

\subsection{\textbf{The Estimation of  Minimal Eigenvalue}}\label{33}
In the start stage of this section, we state a key lemma which is modified from \cite{lilam13}.
Now, for the set $U_x$ we have constructed, define a random process $g_{x}$ by
\begin{equation*}
g_{x}(i)\triangleq{I_{\lbrace Y_i\in U_x\rbrace}-P(Y_i\in U_x|\mathcal{F}_{i-1}^{y})},\quad i\geq 1,
\end{equation*}
where $Y_{i}\triangleq (y_{i+n-1},\ldots,y_{i})^{\tau}$ and $\mathcal{F}_{i-1}^{y}\triangleq \sigma\lbrace \theta, y_0,\ldots,y_{i-1}\rbrace$.
\begin{lemma}\label{gxwy}
For any $\epsilon>0$, there is a class $\mathcal{G}_{\epsilon}$ such that\\
(i) each element of $\mathcal{G}_{\epsilon}$, denoted by $g_{\epsilon}$, is a random series $\lbrace g_{\epsilon}(i)\rbrace_{i\geq 1}$ with the form
\begin{equation}\label{giwy}
g_{\epsilon}(i)=I_{\lbrace Y_i\in U_{\epsilon}\rbrace}-P( Y_i\in U_{\epsilon}|\mathcal{F}_{i-1}^y)-\epsilon,\quad i\geq 1,
\end{equation}
where $U_{\epsilon}$ is a set in $\mathbb{R}^n$;\\
(ii) $\mathcal{G}_{\epsilon}$ contains a lower process $g_{\epsilon}$ to each $g_{x}$ in the sense that
\begin{equation}\label{giiwy}
g_{\epsilon}(i)\leq g_x(i)\quad \forall i\geq 1.
\end{equation}
\end{lemma}
\begin{proof}
(i) Let $O$ be a closed box contains $\mathcal{S}$. Let $r$ be an integer such that
\begin{eqnarray}
r>2\varepsilon^{-1}\overline{\rho}^nC_0\cdot\ell(O),
\end{eqnarray}
where $C_0$ is defined in \dref{uor} and $\overline{\rho}\triangleq \sup_{x\in\mathbb{R}}\rho(x)$. Let $U_{\epsilon}$ be a union of some boxes taken from  $\mathcal{T}(O,r)$. Hence, for a fixed $U_{\epsilon}$, we can define a random process $g_{\epsilon}$ by \dref{giwy}. Denote $\mathcal{G}_{\epsilon}$ as the class of all such $g_{\epsilon}$.

(ii)
Note that for every $x\in \mathbb{R}^m$ with  $\|x\|=1$, $ U_{x}$ is bounded. Then,
there is a set $U_{\epsilon}\in \mathbb{R}^n$ such that $U_{\epsilon}\subset U_{x}$ and $\Delta U_{\epsilon,x}\triangleq{U_x-U_{\epsilon}}$ falls into a union  of finite boxes $J_1,\ldots,J_{l}\in\lbrace U\in \mathcal{T}(O,r): \partial(U_x)\cap U\not=\emptyset\rbrace$. By \dref{uor}, we obtain
\begin{eqnarray}\label{jc1}
\sum_{k=1}^{l}\ell(J_{k})&=& l\cdot \frac{\ell(O)}{r^{n}}\leq C_0\cdot \ell(O)\cdot \frac{1}{r}<\frac{1}{2\overline{\rho}^n}\epsilon.
\end{eqnarray}
We now calculate $P(Y_t\in \Delta U_{\epsilon,x}|\mathcal{F}_{t-1}^y)$. By \dref{jc1},  Lemma \ref{inf>0}(i) and Assumption A1', it is easy to see
\begin{eqnarray*}\label{eq1wy}
&&P(Y_t\in \Delta U_{\epsilon,x}|\mathcal{F}_{t-1}^y)\leq P(Y_t\in \bigcup_{k=1}^{l}J_{k}|\mathcal{F}_{t-1}^y)\nonumber\\
&\leq &\ell(\lbrace (w_{t+n-1},\ldots,w_{t})^{\tau}:  Y_{t}\in \bigcup_{k=1}^{l}J_{k}\rbrace)\cdot \overline{\rho}^n\nonumber\\
&=&\ell\left(\bigcup_{k=1}^{l}J_{k}\right)\cdot \overline{\rho}^n<\frac{\epsilon}{2}.
\end{eqnarray*}
So, for any $i\geq 1$,
\begin{eqnarray*}
g_{x}(i)&=&I_{\lbrace Y_i\in U_{x}\rbrace}-P( Y_i\in U_{x}|\mathcal{F}_{i-1}^{y})\nonumber\\
&=&I_{\lbrace Y_i\in U_{x}\rbrace}-P( Y_i\in U_{\epsilon}|\mathcal{F}_{i-1}^{y})-P(Y_i\in\Delta U_{\epsilon,x}|\mathcal{F}_{i-1}^{y})\nonumber\\
&\geq & I_{\lbrace Y_i\in U_{\epsilon}\rbrace}-P( Y_i\in U_{\epsilon}|\mathcal{F}_{i-1}^{y})-\epsilon=g_{\epsilon}(i),
\end{eqnarray*}
which is  exactly \dref{giiwy}.
\end{proof}

\begin{proof}[Proof of Proposition \ref{cruc}]
First, recall the definition of $U_x$,
for any $x\in\mathbb{R}^{m}$ with $\|x\|=1$, Lemma \ref{inf>0}(ii) and Assumption A1' yield
\begin{eqnarray}\label{pk1}
&&P(Y_{i}\in U_x|\mathcal{F}_{i-1}^{y})I_{\{\|Y_{i-n}\|\leq M, \|\theta\|\leq K\}}\nonumber\\
&= & P(Y_{i}\in \lbrace y: |\phi^{\tau}(y)x|> \delta^{*}\rbrace\cap \mathcal{S}|\mathcal{F}_{i-1}^{y})I_{\{\|Y_{i-n}\|\leq M, \|\theta\|\leq K\}}\nonumber\\
&\geq &\inf_{\|x\|=1, \|y\|\leq M, \|z\|\leq K}\ell\left(\lbrace \varsigma: |\phi^{\tau}(g(\varsigma,y,z))x|> \delta^{*}, g(\varsigma,y,z)\in \mathcal{S}\rbrace\right)\nonumber\\
&&\cdot \left(\inf_{s\in[-S', S']}\rho(s)\right)^nI_{\{\|Y_{i-n}\|\leq M, \|\theta\|\leq K\}}\triangleq  C_PI_{\{\|Y_{i-n}\|\leq M, \|\theta\|\leq K\}},
\end{eqnarray}
where
\begin{eqnarray*}
S'=K\sup_{\|y\|\leq M+R'}\|\phi(y)\|+R'\quad \mbox{and}\quad R'\triangleq \max_{1\leq i\leq n}\mbox{dist}\left(0,\bigcup_{j=1}^{p_i}S_i^{j}(q)\right).
\end{eqnarray*}

Next,  note that for any $\epsilon>0$ and  $g_{\epsilon}\in\mathcal{G}_{\epsilon}$,   $\lbrace g_{\epsilon}(i)+\epsilon, \mathcal{F}_{i}^{y}\rbrace_{i\geq 1}$   is a martingale difference sequence. Taking account of \cite[Theorem 2.8]{cg91},
\begin{eqnarray*}
\lim_{t\rightarrow+\infty}\frac{\sum_{i=1}^{t}I_{\lbrace \|Y_{i-n}\|\leq M\rbrace}(g_{\epsilon}(i)+\epsilon)}{N_{t}(M)}=0,\quad \mbox{a.s.}\quad \mbox{on } \Omega(M),
\end{eqnarray*}
where $\Omega(M)$ is defined in Theorem \ref{tzz}.
Thanks to the finite number of $U_{\epsilon}$ constrained in  $\mathcal{S}$, it gives
\begin{eqnarray}
\lim_{t\rightarrow+\infty}\inf_{U_{\epsilon}\subset\mathcal{S}}\frac{1}{N_{t}(M)}\sum_{i=1}^{t}I_{\lbrace \|Y_{i-n}\|\leq M\rbrace} g_{\epsilon}(i)=-\epsilon,\quad \mbox{a.s.  on }\Omega(M).\nonumber
\end{eqnarray}
As a result,  Lemma \ref{gxwy}(ii) infers that for some $g_{\epsilon}^{x}\in \mathcal{G}_{\epsilon}$,
\begin{eqnarray*}
&&\liminf_{t\rightarrow+\infty}\inf_{\|x\|=1}\frac{1}{N_{t}(M)}\sum_{i=1}^{t}I_{\lbrace \|Y_{i-n}\|\leq M\rbrace} g_{x}(i)\nonumber\\
&\geq &\liminf_{t\rightarrow+\infty}\inf_{\|x\|=1}\frac{1}{N_{t}(M)}\sum_{i=1}^{t}I_{\lbrace \|Y_{i-n}\|\leq M\rbrace}g_{\epsilon}^{x}(i)\nonumber\\
&\geq & \liminf_{t\rightarrow\infty}\inf_{U_{\epsilon}\subset\mathcal{S}}\frac{1}{N_{t}(M)}\sum_{i=1}^{t}I_{\lbrace \|Y_{i-n}\|\leq M\rbrace} g_{\epsilon}(i)\nonumber\\
&=&-\epsilon,\quad \mbox{a.s.}\quad \mbox{a.s.  on }\Omega(M).
\end{eqnarray*}
Further, by the arbitrariness  of $\epsilon$, we obtain
\begin{equation}\label{infa}
\liminf_{t\rightarrow+\infty}\inf_{\|x\|=1}\frac{1}{N_{t}(M)}\sum_{i=1}^{t}I_{\lbrace \|Y_{i-n}\|\leq M\rbrace} g_{x}(i)\geq 0\quad \mbox{a.s.  on }\Omega(M).
\end{equation}

Finally, by \dref{pk1}--\dref{infa},  if $\e$ is sufficiently small,  then there is a positive  random integer $T$ such that for any unit vector $x\in\mathbb{R}^{m}$ and all $t>T$,
\begin{eqnarray*}
&&\frac{1}{N_{t}(M)}\sum_{i=1}^{t}I_{\lbrace \|Y_{i-n}\|\leq M\rbrace}I_{\lbrace Y_i\in U_x\rbrace}\nonumber\\
&>&\frac{1}{N_{t}(M)}\sum_{i=1}^{t}I_{\lbrace \|Y_{i-n}\|\leq M\rbrace}P(Y_{i}\in U_x|\mathcal{F}_{i-1}^{y})-\frac{C_P}{2}\nonumber\\
&\geq &\frac{C_P}{2},\quad \mbox{a.s.}~\mbox{on}\quad \Omega(M)\cap\lbrace \|\theta\|\leq K\rbrace.
\end{eqnarray*}
Hence, we select $C_{\phi}> n R'$, $U_x\subset \overline{B(0,C_{\phi})}$, for sufficiently large $t$,
\begin{eqnarray*}
\lambda_{\min}(t+1)
&=&\inf_{\|x\|=1}x^{\tau}\left(I_{m}+\sum_{i=0}^{t}\phi_{i}\phi_{i}^{\tau}\right)x\nonumber\\
&\geq &\sum_{i=1}^{t-n+1}I_{\lbrace Y_{i}\in U_x\rbrace}(\phi^{\tau}(Y_i)x)^2\nonumber\\
&\geq &(\delta^{*})^2\sum_{i=1}^{t-n+1}I_{\lbrace Y_{i}\in U_x\rbrace}\nonumber\\
&\geq &\frac{(\delta^{*})^2C_P}{2}(N_{t}(M)-n),\quad \mbox{a.s.}~\mbox{on}\quad \Omega(M)\cap\lbrace \|\theta\|\leq K\rbrace.
\end{eqnarray*}
 Proposition \ref{cruc} is thus proved.
\end{proof}

\appendix
\section{}\label{AppA}
In this appendix, we follow the definitions and symbols in the proof of Lemma \ref{xy} and complete  the estimation details of \dref{gj1}. To this end, define
\begin{eqnarray}
I_{k+1}^{*}&\triangleq&\left\lbrace z_{k+1}: \left(\prod_{i=1}^{k}I_{i}\times z_{k+1} \right)\cap \mathcal{B}(\d)\cap \left(\prod_{i=1}^{k}K_{i}\times z_{k+1}\right)\not=\emptyset  \right\rbrace\nonumber\\
&&\cap I_{k+1}\cap\left(\bigcup_{j=1}^{p_{k+1}}\overline{S_{k+1}^{j}(q)}\right),\quad k\geq 1\nonumber\\
\mathcal{T}^3&\triangleq& \lbrace A\in \mathcal{T}^2: A\cap I_{k+1}^{*}\not=\emptyset\rbrace,\nonumber\\
\mathcal{T}^4&\triangleq&\left\lbrace B\in \mathcal{T}^1: \bigcup_{i=1}^{k}\lbrace z: z_{i}\in L_{i}\rbrace\cap B\not=\emptyset \right\rbrace,\nonumber
\end{eqnarray}
where $\prod_{i=1}^{k+1}I_{i}=O$ is the given closed box in the proof of Lemma \ref{xy}.

\begin{lemma}\label{ap1}
The cardinals of  $I_{k+1}^{*}, \mathcal{T}^3$ and $\mathcal{T}^4$   are bounded by
\begin{eqnarray}
|I_{k+1}^{*}|&\leq &(2p_{k+1}(|L_{k+1}|+2)+2)\prod_{i=1}^{k}(|L_{i}|+p_i),\label{I*}\\
|\mathcal{T}^3|&\leq & 2(2p_{k+1}(|L_{k+1}|+2)+2)\prod_{i=1}^{k}(|L_{i}|+p_i),\nonumber\\
|\mathcal{T}^4| &\leq & 2r^{k-1}\sum_{i=1}^{k}|L_i|,\label{T4}
\end{eqnarray}
\end{lemma}
\begin{proof}
By the definitions of $\mathcal{T}^3$ and $\mathcal{T}^4$, $\mathcal{T}^3\leq 2|I_{k+1}^{*}|$ and \dref{T4} is trivial. So, it suffices to show \dref{I*}. For this, recall the definitions of $K_{i}$ and $L_i$, then for each $i\in[1,n]$, there is a set $\mathcal{P}_{i}$ consisting of some disjoint intervals such that  $|\mathcal{P}_i|\leq |L_i|+p_i$  and  $\bigcup_{I\in \mathcal{P}_i}I=K_i$. As a result, $|\prod_{i=1}^{k}\mathcal{P}_{i}|\leq\prod_{i=1}^{k}(|L_{i}|+p_i)$. For each box $B
\in \prod_{i=1}^{k}\mathcal{P}_{i}$, denote
\begin{eqnarray}
I_{k+1}^{*}(B)&=&\lbrace z_{k+1}: (\prod_{i=1}^{k}I_{i}\times z_{k+1})\cap \mathcal{B}(\delta)\cap (B\times z_{k+1})\not=\emptyset  \rbrace\nonumber\\
&&\cap I_{k+1}\cap\left(\bigcup_{j=1}^{p_{k+1}}\overline{S_{k+1}^{j}(q)}\right).\nonumber
\end{eqnarray}
Since $B\subset \prod_{i=1}^{k} K_i$, it is evident that
\begin{eqnarray}\label{phi=cB}
\sum_{i=1}^k x_{i}^{\tau}\phi^{(i)}\equiv \mbox{constant}\qquad    \mbox{on }  B.
\end{eqnarray}
So, for any  $z_{k+1}\in I_{k+1}^{*}(B)$, arbitrarily taking a $(z_1,\ldots,z_{k})^{\tau}\in \mbox{int}(B)$ infers
$$
(z_1,\ldots,z_{k+1})^\tau\in  \mathcal{B}(\delta).
$$
 Let $\lbrace (z_{1,j},\ldots,z_{k+1,j})^{\tau}\rbrace_{j=1}^{+\infty}$ be a sequence of points in  $(\mbox{int}(B)\times E_{k+1})\cap\lbrace y: \phi^{\tau}(y)x> \delta\rbrace$ and  tend to $(z_{1},\ldots,z_{k+1})^{\tau}$. Then, $\lim_{j\rightarrow+\infty}\|z_{k+1,j}-z_{k+1}\|=0$ and
\begin{eqnarray}\label{xk+1phi}
\quad x_{k+1}^{\tau}\phi^{(k+1)}(z_{k+1,j})>\delta-\sum_{i=1}^{k}x_{i}^{\tau}\phi^{(i)}(z_{i,j})=\delta-\sum_{i=1}^{k}x_{i}^{\tau}\phi^{(i)}(z_{i}).
\end{eqnarray}
Denote
\begin{eqnarray}\label{elt}
\bar \delta=\delta-\sum_{i=1}^{k}x_{i}^{\tau}\phi^{(i)}(z_i),
\end{eqnarray}
so \dref{xk+1phi} implies
\begin{eqnarray}
z_{k+1}\in\partial(\lbrace z: x_{k+1}^{\tau}\phi^{(k+1)}(z)>\bar \delta\rbrace)\cap\bigcup_{j=1}^{p_{k+1}}\overline{S_{k+1}^{j}(q)},\nonumber
\end{eqnarray}

Therefore, applying Lemma \ref{ldfb},
\begin{eqnarray*}
|I_{k+1}^{*}(B)|&\leq &\left|\partial(\lbrace z: x_{k+1}^{\tau}\phi^{(k+1)}(z)>\bar \delta\rbrace)\cap\left(\bigcup_{j=1}^{p_{k+1}}\overline{S_{k+1}^{j}(q)}\right)\right|\nonumber\\
&\leq &2p_{k+1}(|L_{k+1}|+2)+2,
\end{eqnarray*}
and thus
\begin{eqnarray}\label{lpi}
|I_{k+1}^{*}|&\leq &(2p_{k+1}(|L_{k+1}|+2)+2)\left|\prod_{i=1}^{k}\mathcal{P}_{i}\right|\nonumber\\
&\leq &(2p_{k+1}(|L_{k+1}|+2)+2)\prod_{i=1}^{k}(|L_{i}|+p_i),
\end{eqnarray}
which completes the proof.
\end{proof}

\begin{lemma}\label{ap2}
Let Lemma \ref{xy}  hold with    $n=k$.
Then, there is a constant $C_1>0$  depending only on $\phi$ such that
\begin{eqnarray}\label{gj2}
\frac{r}{|I_{k+1}|}\int_{I_{k+1}}\sum_{B\in \mathcal{T}^{1}}I_{\mathcal{Z}(B)}dz_{k+1}
\leq C_1 r^{k}.
\end{eqnarray}
\end{lemma}
\begin{proof}
Denote $\phi'=\mbox{col}\lbrace \phi^{(1)},\ldots,\phi^{(k)}\rbrace$, $x'=\mbox{col}\lbrace x_1,\ldots,x_{k}\rbrace$ and $z=(z_1,\ldots,z_{k})^{\tau}$.
Given $z_{k+1}\in I_{k+1}$, define $\delta'\triangleq \delta-\phi^{(k+1)}(z_{k+1})x_{k+1}$. Then,
\begin{eqnarray*}
&&\lbrace z: (z_1,\ldots,z_{k+1})^{\tau}\in \mathcal{B}(\d) \rbrace \cap \prod_{i=1}^{k} A^c(x_{i}^{\tau}(\phi^{(i)})')\cap \left(\prod_{i=1}^{k}\bigcup_{j=1}^{p_i}\overline{S_i^{j}(q)}\right)  \\
&=&\partial(\lbrace z: (\phi')^{\tau}(z)x'> \delta'\rbrace)\cap\prod_{i=1}^{k} A^c(x_{i}^{\tau}(\phi^{(i)})')\cap\left(\prod_{i=1}^{k}\bigcup_{j=1}^{p_i}\overline{S_i^{j}(q)}\right).
\end{eqnarray*}
In addition, for $\lbrace L_i,  K_i\rbrace_{i=1}^{n}$ defined in Lemma \ref{sdf},
$$
\left(\prod_{i=1}^{k} A^c(x_{i}^{\tau}(\phi^{(i)})')\right)^c=\left(\bigcup_{i=1}^{k}\lbrace z: z_{i}\in L_{i} \rbrace\right)\cup\prod_{i=1}^{k}K_{i},
$$
so we arrive at
\begin{eqnarray}\label{lbr}
&&\lbrace z: (z_1,\ldots,z_{k+1})^{\tau}\in \mathcal{B}(\d) \rbrace \cap\left(\prod_{i=1}^{k}\bigcup_{j=1}^{p_i}\overline{S_i^{j}(q)}\right)\nonumber\\
&\subset &\partial(\lbrace z: (\phi')^{\tau}(z)x'> \delta'\rbrace)\cap\left(\prod_{i=1}^{k}\bigcup_{j=1}^{p_i}\overline{S_i^{j}(q)}\right)\nonumber\\
&& \cup\left(\bigcup_{i=1}^{k}\lbrace z: z_{i}\in L_{i} \rbrace\right)\cup\prod_{i=1}^{k}K_{i}.
\end{eqnarray}
Consequently,  for any $z_{k+1}\in A\in \mathcal{T}^2\setminus \mathcal{T}^3$ and  $B\in \mathcal{T}^{1}\setminus \mathcal{T}^4$, \dref{lbr} shows
\begin{eqnarray}
&&\lbrace z: (z_1,\ldots,z_{k+1})^{\tau}\in\mathcal{B}(\d)\rbrace\cap\prod_{i=1}^{k}\bigcup_{j=1}^{p_i}\overline{S_i^{j}(q)}\cap B\nonumber\\
&\subset &\partial(\lbrace z: (\phi')^{\tau}(z)x'> \delta'\rbrace)\cap\prod_{i=1}^{k}\bigcup_{j=1}^{p_i}\overline{S_i^{j}(q)}\cap B.\nonumber
\end{eqnarray}

Now, for $\partial(\lbrace z: (\phi')^{\tau}(z)x'> \delta'\rbrace)\cap\prod_{i=1}^{k}\bigcup_{j=1}^{p_i}\overline{S_i^{j}(q)}$ and $\mathcal{T}^1$,   applying Lemma \ref{xy} with $n=k$ 
leads to
\begin{eqnarray}\label{bt1}
\sum_{B\in \mathcal{T}^{1}\setminus \mathcal{T}^4}I_{Z(B)}(z_{k+1})\leq C r^{k-1}.
\end{eqnarray}
Based on \dref{bt1}, it is readily to compute
\begin{eqnarray}\label{ibt1}
&&\int_{I_{k+1}}\sum_{B\in \mathcal{T}^{1}}I_{\mathcal{Z}(B)}dz_{k+1}=\sum_{A\in \mathcal{T}^2}\int_{A}\sum_{B\in \mathcal{T}^{1}}I_{\mathcal{Z}(B)}dz_{k+1}\nonumber\\
&\leq &\sum_{A\in \mathcal{T}^2\setminus \mathcal{T}^3}\int_{A}\sum_{B\in \mathcal{T}^{1}}I_{\mathcal{Z}(B)}dz_{k+1}+  \sum_{A\in \mathcal{T}^3}\int_{A}r^kdz_{k+1}         \nonumber\\
&=&\sum_{A\in \mathcal{T}^2\setminus \mathcal{T}^3}\int_{A}\sum_{B\in \mathcal{T}^{1}\setminus \mathcal{T}^4}I_{\mathcal{Z}(B)}dz_{k+1}+\sum_{A\in \mathcal{T}^2\setminus \mathcal{T}^3}\int_{A}\sum_{B\in \mathcal{T}^4}I_{\mathcal{Z}(B)}dz_{k+1}\nonumber\\
&&+r^{k}\cdot\frac{|I_{k+1}|}{r}\cdot|\mathcal{T}^3|\nonumber\\
&\leq &\int_{I_{k+1}}Cr^{k-1}dz_{k+1}+
\sum_{B\in \mathcal{T}^4}\int_{I_{k+1}}1 dz_{k+1}+r^{k-1}|I_{k+1}||\mathcal{T}^3|.\nonumber\\
&\leq & ((C+|\mathcal{T}^3|)r^{k-1}+|\mathcal{T}^4|)
|I_{k+1}|.\nonumber
\end{eqnarray}
The result follows from  Lemmas  \ref{ap1} and  \ref{sdf}.
\end{proof}

\begin{lemma}\label{ap3}
There is a constant $C_2>0$ depends only  on $\phi$ such that
\begin{eqnarray}
\sum_{B\in \mathcal{T}^{1}}|\mathcal{Z}_{2}(B)|\leq C_2 r^{k}.\nonumber
\end{eqnarray}
\end{lemma}
\begin{proof}
Let
\begin{eqnarray}
\mathcal{T}^5\triangleq \left\lbrace \prod_{i=1}^{k}I_{i}'\in \mathcal{T}^{1}:  \partial\left(\bigcup_{j=1}^{p_i}S_i^{j}(q)\right)\cap I_{i}'\not=\emptyset \mbox{ for some } i\in [1,k] \right\rbrace.\nonumber
\end{eqnarray}
Clearly,  $|\mathcal{T}^5|\leq 4r^{k-1}\sum_{i=1}^{k}p_{i}$. Hence,
\begin{eqnarray}\label{Z2B}
\sum_{B\in \mathcal{T}^{1}}|\mathcal{Z}_{2}(B)|\leq \sum_{B\in \mathcal{T}^{1}\setminus (\mathcal{T}^5\cup \mathcal{T}^4)}|\mathcal{Z}_{2}(B)|+r|\mathcal{T}^4|+4r^{k}\sum_{i=1}^{k}p_{i}.
\end{eqnarray}
It suffices to estimate the first term in the right hand side of \dref{Z2B}. To this end,
take a set $B=\prod_{i=1}^{k}I'_{i}\in \mathcal{T}^{1}\setminus (\mathcal{T}^5\cup \mathcal{T}^4)$ and let $z_{k+1}\in \partial Z(B)\cap\mbox{int}(I_{k+1})$.   Select a point $(z_{1},\ldots,z_{k})^{\tau}\in B$  that
\begin{eqnarray}\label{z1zkdef}
&&\mbox{dist}((z_{1},\ldots,z_{k+1})^{\tau},\prod_{i=1}^{k}\partial(I'_{i})\times z_{k+1})\nonumber\\
&=&\min_{y\in\mathcal{B}(\delta)\cap \prod_{i=1}^{k+1}\bigcup_{j=1}^{p_i}\overline{S_i^{j}(q)}\cap(B\times z_{k+1})}\mbox{dist}(y,\prod_{i=1}^{k}\partial(I'_{i})\times z_{k+1}).
\end{eqnarray}
Clearly, $B\in \mathcal{T}^{1}\setminus (\mathcal{T}^5\cup \mathcal{T}^4)$ implies that for each $i=1,\ldots,k$,
$$\mbox{int}(I'_{i})\subset \bigcup_{j=1}^{p_i}S_i^{j}(q)\quad \mbox{and}\quad \mbox{int}(I'_{i})\cap L_i=\emptyset.
$$
We consider the following two cases:
\\
\emph{Case 1:} $(z_{1},\ldots,z_{k})^{\tau}\not\in\prod_{i=1}^{k}\partial(I'_{i})$. Then, there is an integer $i\in[1,k]$ such that $z_{i}\in\mbox{int}(I'_{i})$.
By \dref{z1zkdef},  $z_{i}\not\in K_{i}\cap\mbox{int}(I'_{i})$. Otherwise, there is a $\rho>0$    such that
 $x_i^{\tau}(\phi^{(i)})'\equiv0$ on $[z_i-\rho,z_{i}+\rho]\subset \mbox{int}(I'_i)$.
Similar to \dref{phi=cB}--\dref{xk+1phi},  for any $z'_i\in [z_i-\rho,z_{i}+\rho]$,
\begin{eqnarray}
(z_{1},\ldots, z_{i-1}, z'_i, z_{i+1},\ldots,  z_{k+1})^{\tau}\in  \mathcal{B}(\delta)\cap \prod_{i=1}^{k+1}\bigcup_{j=1}^{p_i}\overline{S_i^{j}(q)}\cap(B\times z_{k+1}).\nonumber
\end{eqnarray}
Then,
\begin{eqnarray}
&&\min\bigg\lbrace \mbox{dist}((z_{1},\ldots, z_{i-1}, z_i-\rho, z_{i+1},\ldots,  z_{k+1})^{\tau},\prod_{i=1}^{k}\partial(I'_{i})\times z_{k+1})\nonumber\\
&& ~~~~~~~~~~\mbox{dist}((z_{1},\ldots, z_{i-1}, z_i+\rho, z_{i+1},\ldots,  z_{k+1})^{\tau},\prod_{i=1}^{k}\partial(I'_{i})\times z_{k+1})\bigg\rbrace\nonumber\\
&<& \mbox{dist}((z_{1},\ldots,z_{k+1})^{\tau},\prod_{i=1}^{k}\partial(I'_{i})\times z_{k+1}),\nonumber
\end{eqnarray}
which contradicts to \dref{z1zkdef}.

Now, since $z_{i}\not\in K_{i}\cap\mbox{int}(I'_{i})$ and $B\notin \mathcal{T}^4$,
it yields that $x_i^{\tau}(\phi^{(i)})'(z_{i})\not=0$. We claim
\begin{eqnarray}\label{z1}
z_{k+1}\in\bigcup_{j=1}^{p_{k+1}}\partial(S_i^{j}(q)).
\end{eqnarray}
 Otherwise, $z_{k+1}\in\bigcup_{j=1}^{p_{k+1}}S_i^{j}(q)$. By the \textit{Implicit function theorem}, there is a sufficiently small $\eta>0$ such that for any $z_{k+1}'\in(z_{k+1}-\eta,z_{k+1}+\eta)$, a point $z_{i}'\in\mbox{int}(I_{i})$ exists and
\begin{eqnarray*}
(z_1,\ldots,z_{i-1},z_i',z_{i+1},\ldots,z_{k},z_{k+1}')^{\tau}\in \mathcal{B}(\delta)\cap \prod_{i=1}^{k+1}\bigcup_{j=1}^{p_i}\overline{S_i^{j}(q)}.
\end{eqnarray*}
This means   $z_{k+1}\in \mbox{int}(Z(B))$, which is impossible due to  $z_{k+1}\in \partial Z(B)$. Hence \dref{z1} holds.

\emph{Case 2:} $(z_{1},\ldots,z_{k})^{\tau}\in\prod_{i=1}^{k}\partial(I'_{i})$.
 Since $z_{k+1}\in \partial(Z(B))$, $x_{k+1}^{\tau}\phi^{(k+1)}$ cannot be a constant on any neighbourhood of $z_{k}$. So,
\begin{eqnarray}\label{z2}
z_{k+1} &\in &\partial(\lbrace z: x_{k+1}^{\tau}\phi^{(k+1)}(z)\not=\bar \delta\rbrace)\cap\left(\bigcup_{j=1}^{p_{k+1}}S_i^{j}(q)\right)\nonumber\\
&&\cup \left(\bigcup_{j=1}^{p_{k+1}}\partial(S_i^{j}(q))\right),
\end{eqnarray}
where $\bar \delta$ is defined by \dref{elt}.

Combining  the above two cases, $z_{k+1}\in \partial(Z(B))\cap \mbox{int}(I_{k+1})$ implies \dref{z2}.
Taking the case $z_{k+1}\in \partial(I_{k+1})$ into consideration, we obtain
\begin{eqnarray}
\partial(Z(B))&\subset &\partial(\lbrace y\in \mathbb{R}: x_{k+1}^{\tau}\phi^{(k+1)}(y)\not=\bar \delta\rbrace)\cap\left(\bigcup_{j=1}^{p_{k+1}}S_i^{j}(q)\right)\nonumber\\
&&\cup \left(\bigcup_{j=1}^{p_{k+1}}\partial(S_i^{j}(q))\right)\cup \partial(I_{k+1}).
\end{eqnarray}
which, together with the fact $|\partial(\lbrace z: x_{k+1}^{\tau}\phi^{(k+1)}(z)\not=\bar \delta\rbrace)\cap(\bigcup_{j=1}^{p_{k+1}}S_i^{j}(q))|\leq 4p_{k+1}(|L_{k+1}|+2)$ from \dref{bd2}, leads to
\begin{eqnarray}
|\mathcal{Z}_{2}(B)|\leq 2|\partial(Z(B))|\leq 8p_{k+1}(|L_{k+1}|+2)+4p_{k+1}+4.\nonumber
\end{eqnarray}
Now, in view of \dref{Z2B}, we derive
\begin{eqnarray}
\sum_{B\in \mathcal{T}^{1}}|\mathcal{Z}_{2}(B)|\leq (8p_{k+1}(|L_{k+1}|+2)+4p_{k+1}+4) r^{k}+|\mathcal{T}^4| r+4r^{k}\sum_{i=1}^{k}p_{i},\nonumber
\end{eqnarray}
which yields the result by Lemma \ref{ap1}.
\end{proof}

\section{}\label{AppB}
In this appendix, we provide the proof of Theorems \ref{tzz2} by showing
\begin{proposition}\label{cruc1}
Proposition \ref{cruc} holds for model \dref{sys} if Assumption A2 is replaced by A3.
\end{proposition}
\begin{proof}[Proof of Proposition \ref{cruc1}]
The proof is similar as that of Proposition \ref{cruc} but more concise due to Assumption A3.
First, we need not to construct $\mathcal{S}$ from Lemmas \ref{useful}--\ref{qq}. As a matter of fact, taking  $\delta^{*}$  from \dref{dl*} in Assumption A3, Lemma \ref{inf>0}  follows with $\mathcal{S}$ replaced by $E$.
 So, for every unit vector $x\in\mathbb{R}^{m}$, we can directly define
$$
U_x\triangleq\lbrace y: |\phi^{\tau}(y)x|>\delta^{*} \rbrace\cap E.
$$

Next, with random process $g_x$ defined in Subsection \ref{33}, we proceed to Lemma \ref{gxwy}. To show this lemma in the current case, we are not going to verify \dref{jc1}   by using Lemmas \ref{imp}--\ref{xy}. Instead,  we intend to claim another formula. For this,
select a box $O$ containing $\overline{E}$ and define
\begin{eqnarray}\label{kdef1}
\mathcal{T}(x,O,r)  \triangleq \left\{ U\in \mathcal{T}(O,r): \partial{U_x} \cap U \neq \emptyset \right   \},
\end{eqnarray}
where $\mathcal{T}(O,r)$ is defined above \dref{kdef}. The remainder is mainly devoted to proving
\begin{eqnarray}\label{suu}
\lim_{r\rightarrow+\infty}\sup_{\|x\|=1}\sum_{U\in \mathcal{T}(x,O,r) }\ell(U)=0.
\end{eqnarray}

To show \dref{suu}, note that
\begin{eqnarray}
\partial(U_x)\subset V_x\triangleq \lbrace y\in \overline{E}: |\phi^{\tau}(y)x|=\delta^{*}\rbrace.
\end{eqnarray}
Denote $W(x,r)\triangleq\bigcup_{U\in \mathcal{T}'(x,O,r)}U$, where  
\begin{eqnarray}\label{kdef2}
\mathcal{T}'(x,O,r)  \triangleq \left\{ U\in \mathcal{T}(O,r): V_x \cap U \neq \emptyset \right   \}.
\end{eqnarray}
So,
it suffices to show
\begin{eqnarray}\label{suu1}
\lim_{r\rightarrow+\infty}\sup_{\|x\|=1}\ell(W(x,r))=0.
\end{eqnarray}
If \dref{suu1} is false, then there is a number $\varepsilon>0$ and a unit vector sequence $\lbrace x(k)\rbrace_{k=1}^{+\infty}$ such that $\lim_{k\rightarrow+\infty}x(k)=x^{*}$ for some unit vector $x^{*}$ and
\begin{eqnarray}\label{W>e}
\ell(W(x(k),2^k))>\varepsilon,\quad \forall k\geq 1.
\end{eqnarray}
Now, according to the definition of the Jordan measure,
 \dref{Jordan} in  Assumption A3(ii) indicates that
$
\lim_{r\rightarrow+\infty}\ell(W(x^{*},r))=0.$ Moreover,
since
$$
\lim_{k\rightarrow+\infty}\sup_{y\in V_{x(k)}}\mbox{dist}(y,V_{x^{*}})=0,
$$
 for any $\varepsilon'>0$ and all sufficiently large integers $k',k$ with $k'<k$,
\begin{eqnarray}
|\mathcal{T}'(x^{*},O,2^{k})|<\frac{\varepsilon' 2^{kn}}{\ell(O)}\nonumber
\end{eqnarray}
and
\begin{eqnarray}
|\mathcal{T}'(x(k),O,2^{k})|<(1+2^{k-k'+1})^n|\mathcal{T}'(x^{*},O,2^{k})|.\nonumber
\end{eqnarray}
The above two inequalities immediately lead to
\begin{eqnarray}
\ell(W(x(k),2^k))=\frac{\ell(O)}{2^{kn}}\cdot|\mathcal{T}'(x(k),O,2^{k})|<(1+2^{k-k'+1})^n\varepsilon',\nonumber
\end{eqnarray}
which contradicts to \dref{W>e} by  selecting $k'=k-1$ and $\varepsilon'<5^{-n}\varepsilon$.

Finally, \dref{jc1} follows from \dref{suu} and hence  Lemma \ref{gxwy} holds. The rest of the procedures
thus keep the same as  those for Proposition \ref{cruc}.
\end{proof}



\end{document}